\documentclass[12pt,reqno]{amsart}

\usepackage{amsmath}
\usepackage{color}
\usepackage{amsthm}
\usepackage{graphicx}
\usepackage{lscape}
\usepackage[all]{xy}
\usepackage[normalem]{ulem}
\usepackage{adjustbox}

\usepackage{epsfig}

\usepackage{amssymb}
\usepackage{latexsym}

\newtheorem{Theorem}{Theorem}
\newtheorem*{Theorem*}{Theorem}

\newtheorem{Lemma}{Lemma}

\newtheorem{case}{Case}

\theoremstyle{definition}

\def\OO{{\mathcal O}}

\def\O{{\mathrm{O}}}
\def\SO{\mathrm{SO}}
\def\PO{\mathrm{PO}}

\def\mod{{\mathrm{mod}}}

\def\Diff{\mbox{Diff}}
\def\Isom{\mbox{Isom}}
\def\Isompr{\mbox{Isom}\!^+\!}
\def\Out{\mbox{Out}}
\def\Aut{\mbox{Aut}}
\def\Inn{\mbox{Inn}}
\def\Norm{\mbox{Norm}}

\def\C{{\mathbb C}}
\def\R{{\mathbb R}}
\def\H{{\mathbb H}}
\def\Z{{\mathbb Z}}

\begin{document}

\title[Isometry groups and mapping class groups of spherical 3-orbifolds]{Isometry groups and mapping class groups\\of spherical 3-orbifolds}

\author[M. Mecchia]{Mattia Mecchia*}
\address{M. Mecchia: Dipartimento Di Matematica e Geoscienze, Universit\`{a} degli Studi di Trieste, Via Valerio 12/1, 34127, Trieste, Italy.} \email{mmecchia@units.it}
\author[A. Seppi]{Andrea Seppi**}
\address{A. Seppi: CNRS and Universit\'e Grenoble Alpes, 100 Rue des Math\'ematiques, 38610 Gi\`eres, France.} \email{andrea.seppi@univ-grenoble-alpes.fr}

%\subjclass{}
\thanks{$^{*}$Partially supported by the FRA 2015 grant  ``Geometria e topologia delle variet\`{a} ed applicazioni'', Universit\`{a} degli Studi di Trieste and by the ``National Group for Algebraic and Geometric Structures, and their Applications'' (GNSAGA€" INdAM)}

\thanks{$^{**}$Partially supported by the FIRB 2011-2014 grant ``Geometry and topology of low-dimensional manifolds''  and by the ``National Group for Algebraic and Geometric Structures, and their Applications'' (GNSAGA€" INdAM)}

%\date{\today}

\begin{abstract}
We study the isometry groups of compact spherical orientable $3$-orbifolds $S^3/G$, where $G$ is a finite subgroup of $\mathrm{SO}(4)$, by determining their isomorphism type. Moreover, we prove that the inclusion of $\mbox{Isom}(S^3/G)$ into $\mbox{Diff}(S^3/G)$ induces an isomorphism of the $\pi_0$ groups, thus proving the $\pi_0$-part of the natural generalization of the Smale Conjecture to spherical $3$-orbifolds.
\end{abstract}

\maketitle
  
\section{Introduction} %-----------------------INTRO---------------------------

Orbifolds are a generalization of manifolds, which had been introduced in different contexts by Satake \cite{satake}, by Thurston \cite[Chapter 13]{thurston2} and by Haefliger \cite{haefliger} -- useful references being also \cite{boileau-maillot-porti, choi, Dun2, scott}. The most standard example of an orbifold (of dimension $n$) is the quotient of a manifold $M^n$ by a group $\Gamma$ which acts properly discontinuously -- but in general not freely -- on $M$. If the action is not free, \emph{singular points} appear in the quotient $M/\Gamma$, keeping track of the action of point stabilizers $\mathrm{Stab}_\Gamma(x)$ on a neighborhood of a fixed point $x\in M$. More generally, an orbifold is \emph{locally} the quotient of a manifold by the action of a finite group.

\emph{Geometric} $3$-\emph{orbifolds} had a large importance in Thurston's geometrization program. These are \emph{locally} the quotient of one of the eight Thurston's model geometries by the properly discontinuous action of a group of isometries. The main object of this paper are compact \emph{spherical} $3$-orbifolds, which are \emph{globally} the quotient of the $3$-sphere $S^3$ by the action of a finite group $G$ of isometries. Hence the quotient orbifold  inherits a metric structure (which is a Riemannian metric of constant curvature $1$ outside the singularities). The main purpose of this paper is to study the group of \emph{isometries} of compact spherical $3$-orbifolds $\mathcal O=S^3/G$, for $G$ a finite subgroup of $\SO(4)$. Roughly speaking, an isometry of $\mathcal O$ is a diffeomorphism which preserves both the induced metric and the type of singularities.

A widely studied problem concerning the isometry group of $3$-manifolds is the \emph{Smale Conjecture}, and its stronger version, called \emph{Generalized Smale Conjecture}. The latter asserts that the natural inclusion of $\Isom(M)$, the group of isometries of a compact spherical $3$-manifold $M$, into $\Diff(M)$ (its group of diffeomorphisms) is a homotopy equivalence. The original version was stated for $M=S^3$ by Smale. The $\pi_0$-\emph{part} of the original conjecture, namely the fact that the natural inclusion induces a bijection on the sets of path components, was proved by Cerf in \cite{cerf}. The full conjecture was then proved by Hatcher in \cite{hatcher}. The Generalized Smale Conjecture for spherical $3$-manifolds was proven in many cases, but is still open in full generality \cite{MR2976322}. The $\pi_0$-{part} was instead proved in \cite{mccullough}. We will prove the $\pi_0$-{part} of the analogous statement for spherical $3$-orbifolds, namely:

\begin{Theorem*}[$\pi_0$-part of the Generalized Smale Conjecture for spherical 3-orbifolds]
Let $\OO=S^3/G$ be a compact spherical oriented orbifold. The inclusion $\Isom(\OO)\rightarrow\Diff(\OO)$ induces a group isomorphism
$$\pi_0 \Isom(\OO)\cong \pi_0 \Diff(\OO)\,.$$
\end{Theorem*}

The proof uses both the algebraic description of the finite groups $G$ acting on $S^3$ by isometries, and the geometric properties of \emph{Seifert fibrations for orbifolds}. In the specific case  of spherical 3-orbifolds whose  underlying topological space is the 3-sphere and whose singular set is a  Montesinos link, the same  result was  proved in \cite{sakuma}.

In fact, the classification of spherical $3$-orbifolds up to orientation-preserving isometries is equivalent to the classification of finite subgroups of $\SO(4)$ up to conjugacy in $\SO(4)$. Such algebraic classification was first given by
Seifert and Threlfall in \cite{threlfall-seifert} and \cite{threlfall-seifert2}; we will use the approach of  \cite{duval}.
In the spirit of the paper \cite{mccullough}, we provide an algebraic description of the isometry group of the spherical orbifold $S^3/G$ once the finite subgroup $G<\SO(4)$ is given. 
To perform the computation, we first understand the group of orientation-preserving isometries, which is isomorphic to the quotient of the normalizer of $G$ in $\SO(4)$ by $G$ itself. We then describe the full group of isometries, when $S^3/G$ has orientation-reversing isometries. Only part of the understanding of $\Isom(S^3/G)$ is indeed necessary for the proof of the $\pi_0$-{part} of the Generalized Smale Conjecture for orbifolds, but to the opinion of the authors it is worthwhile to report the isomorphism type of the isometry group for every spherical orbifold, as such list is not available in the literature. Moroever, by the main theorem above, this gives also the computation for the mapping class group of spherical orbifolds.

To prove that the homomorphism $\iota:\pi_0 \Isom(\OO)\to\pi_0 \Diff(\OO)$ induced by the inclusion $\Isom(\OO)\rightarrow\Diff(\OO)$ is an isomorphism, we will in fact prove that the composition 
\[
\xymatrix{
\pi_0\Isompr(\OO) \ar[r]^-{\iota} & \pi_0\Diff^+\!(\OO) \ar[r]^-{\alpha} & \Out(G)\,,
}
\]
is injective, where $\Out(G)=\Aut(G)/\Inn(G)$. 
This implies that $\iota$ is injective, while surjectivity follows from \cite{cuccagna-zimmermann}. However, the injectivity of $\alpha\circ \iota$ does not hold in general, but will be proved when the singular locus of $\OO$ is nonempty and its complement is a Seifert fibered manifold, while the remaining cases were already treated in \cite{mccullough} and \cite{cuccagna-zimmermann}. In order to detect which orbifolds have this property and to prove injectivity in those cases, it is necessary to analyze the Seifert fibrations which a spherical orbifolds $S^3/G$ may admit, in the general setting of Seifert fibrations for orbifolds.   %in terms of the invariants of the fibration as introduced in \cite{bonahon-siebenmann}, 
The methods to obtain such analysis were provided in \cite{mecchia-seppi}. 

%{\color{red}{\sout{Motivated by the relevance of the notion of Seifert fibration for spherical orbifolds, we decided to include in this paper a more explicit discussion of the action of the isometry group $\Isom(S^3/G)$ on Seifert fibrations of $S^3/G$ induced from some isometric copy of the standard Hopf fibration of $S^3$ (when any exists). We thus give a geometric interpretation of the action of the subgroup of $\Isom(S^3/G)$ which preserves a Seifert fibration (such subgroup being always orientation-preserving, for general reasons). }}}

\subsection*{Organization of the paper} In Section \ref{sec spherical 3-orbifolds}, we explain the algebraic classification of finite subgroups of $\SO(4)$ up to conjugacy, and we give an introduction of two- and three-dimensional orbifolds, with special attention to the spherical case. In Section \ref{sec isometry group}, we compute the isometry group of spherical $3$-orbifolds, by first computing the subgroup of orientation-preserving isometries and then the full isometry group. The results are reported in Tables \ref{tableisometry}, \ref{tableisometryexceptions} and \ref{tableisometryreversing}.   Finally, in Section \ref{sec gen smale} we prove the $\pi_0$-part of the Generalized Smale Conjecture for spherical orbifolds, {{after introducing the necessary background on Seifert fibrations.}}

\subsection*{Acknowledgements} {{We would like to thank an anonymous referee for several advices which improved the exposition of the paper.}}

\section{Spherical three-orbifolds} \label{sec spherical 3-orbifolds}

Let    $\H=\{a+bi+cj+dk\,|\,a,b,c,d\in\R\}=\{z_1+z_2j\,|\,z_1,z_2\in\C\}$ be the quaternion algebra. Given a quaternion $q=z_1+z_2 j$, we denote by $\bar q=\bar z_1-z_2 j$ its conjugate. Thus $\H$ endowed with the positive definite quadratic form given by $q\bar q=|z_1|^2+|z_2|^2$, is isometric to the standard scalar product on $\R^4$.
We will consider  the round 3-sphere  as the set of unit quaternions:

\begin{equation*}
S^3=\{a+bi+cj+dk \,|\, a^2+b^2+c^2+d^2=1\}=\{z_1+z_2j\, \,|\, |z_1|^2+|z_2|^2=1\}\,.
\end{equation*}
The restriction of the product of $\mathbb{H}$ induces a group structure on $S^3$. For $q\in S^3$, $q^{-1}=\bar q$.

\subsection{Finite subgroups of $\SO(4)$} \label{subsec finite subgroups}
In this subsection we present  the classification of the finite subgroup of $\SO(4)$, which is originally due to Seifert and Threlfall (\cite{threlfall-seifert} and \cite{threlfall-seifert2}). More details can be found in \cite{duval,conway-smith,mecchia-seppi}; we follow the approach and the notation of \cite{duval}. We have to mention that in Du Val's list of  finite subgroup of $\SO(4)$ there are three missing cases.

Let us consider the group homomorphism $$\Phi:S^3\times S^3\rightarrow \SO(4)\,$$ which associates to the pair $(p,q)\in S^3\times S^3$ the map $\Phi_{p,q}:\H \rightarrow \H$  with 
$$\Phi_{p,q}(h)=phq^{-1}\,,$$ which is an isometry of $S^3$. The homomorphism $\Phi$ can be proved to be surjective and has kernel $$\mathrm{Ker}(\Phi)=\{(1,1),\,(-1,-1)\}\,.$$
Therefore $\Phi$ gives a  1-1 correspondence between  finite subgroups of $\SO(4)$ and finite subgroups of $S^3\times S^3$ containing the kernel of $\Phi$. 
Moreover, if two subgroups are conjugate in $\SO(4)$, then the corresponding groups in $S^3\times S^3$ are conjugate and vice versa.   
To give a classification  of finite subgroups of $\SO(4)$ up to conjugation, one can thus classify the subgroups of $S^3\times S^3$ which contain $\{(1,1),\,(-1,-1)\}$, up to conjugation in $S^3\times S^3$. 

Let  $\tilde G$ be  a finite subgroup of $S^3\times S^3$ and let us denote by  $\pi_i:S^3\times S^3 \rightarrow S^3$, with $i=1,2$, the two projections. We use the following notations:  $L=\pi_1(\tilde G)$, $L_K=\pi_1((S^3\times\{1\})\cap \tilde G)$, $R=\pi_2(\tilde G)$, $R_K=\pi_2((\{1\}\times S^3)\cap \tilde G)$. The projection $\pi_1$ induces an isomorphism   $$\bar{\pi}_1: \tilde G/(L_K\times R_K)\rightarrow L/L_K\,,$$ 
and $\pi_2$ induces an isomorphism   $$\bar{\pi}_2: \tilde G/(L_K\times R_K)\rightarrow R/R_K\,.$$ 
Let us denote by $\phi_{\tilde G}:L/L_K\to R/R_K$ the isomorphism  
$$\phi_{\tilde G}=\bar{\pi}_1^{-1}\circ \bar{\pi}_2\,.$$ 
On the other hand, if we consider  two finite subgroups $L$ and $R$ of $S^3$, with two normal subgroups $L_K$ and $R_K$ such that there exists an isomorphism $\phi:L/L_K\rightarrow R/R_K$, we can define a subgroup $\tilde G$ of $S^3\times S^3$ such that $L=\pi_1(\tilde G)$, $L_K=\pi_1((S^3\times\{1\})\cap \tilde G)$, $R=\pi_2(\tilde G)$, $R_K=\pi_2((\{1\}\times S^3)\cap \tilde G)$ and $\phi=\phi_{\tilde G}$.
The subgroup $\tilde G$ of $S^3\times S^3$ is determined uniquely by the 5-tuple $(L,L_K,R,R_K,\phi)$.

To consider the classification up to conjugacy, one uses the following straightforward lemma, which is implicitly used in \cite{duval}.

%\begin{Proposition}\label{classificationS3}
%Let $G$ and $G'$ two groups of $S^3\times S^3$ described respectively by $(L,L_K,R,R_K,\phi)$ and  $(L',L'_K,R',R'_K,\phi')$.
%The groups $G$ and $G'$ are conjugated in $S^3\times S^3$ if and only if there exist  two inner automorphisms, $\alpha$ and $\beta$, of $S^3$ such that $\alpha(L)=L'$, $\beta(R)=R'$, $\alpha(L_K)=L'_K$, $\beta(R_K)=R'_K$ and $\phi=\bar{\beta}^{-1}\phi'\bar{\alpha}$, where $\bar{\alpha}$ and $\bar{\beta}$ are the maps induced by $\alpha$ and $\beta$ on the factors $L/L_K$ and $R/R_K$.    
%\end{Proposition}

\begin{Lemma}\label{classificationS3}
Let $\tilde G=(L,L_K,R,R_K,\phi)$ and $\tilde G'=(L',L'_K,R',R'_K,\phi')$ be finite subgroups of $S^3\times S^3$ containing $\mathrm{Ker}(\Phi)$. An element $(g,f)\in S^3\times S^3 $ conjugates $\tilde G$ to $\tilde G'$ if and only if the following three conditions are satisfied:
\begin{enumerate}
\item $g^{-1}Lg=L'$ and $f^{-1}Rf=R'$;
\item $g^{-1}L_Kg=L_K'$ and $f^{-1}R_Kf=R_K'$;
\item the following diagram commutes: 

\begin{equation} %\label{diagramma quadrato}
\begin{gathered}
\xymatrix{
L/L_{K} \ar[d]^-{\alpha}\ar[r]^-{\phi} & R/R_{K} \ar[d]^-{\beta}\\
L'/L'_{K}   \ar[r]^-{\phi'} & R'/R'_{K} \\
}
\end{gathered}
\end{equation}
where $\alpha(xL_K)=g^{-1}xgL'_K$ and   $\beta(yR_K)=f^{-1}yfR'_K$. 
\end{enumerate}
   
\end{Lemma}
%\begin{proof} The proof is straightforward.
%\end{proof} 

Observe that the diagonal subgroup $\Delta$ in $S^3\times S^3$ is the subgroup which preserves the antipodal points $1$ and $-1$, and thus also preserves the equatorial $S^2$ which is equidistant from $1$ and $-1$. Thus one obtains a map
$$\Phi:\Delta\cong S^3\to \SO(3)$$
which associates to $q\in S^3$ the isometry $h\mapsto qhq^{-1}$. By means of this map, and the classification of finite subgroups of $\SO(3)$ one shows that the finite subgroups of $S^3$ are:

$$
\begin{array}  {rll}
C_n=& \{\cos\left(\frac{2\alpha\pi}{n}\right)+i\sin\left(\frac{2\alpha\pi}{n}\right) \,|\,\alpha=0,\dots,n-1\} & n\geq 1  \\[5pt]
D^*_{2n}=& C_n\cup C_n j & n\geq 3 \\[5pt]
T^*=&\bigcup\limits_{r=0}^2 (\frac{1}{2}+\frac{1}{2}i+\frac{1}{2}j+\frac{1}{2}k)^r D^*_{4} &  \\[5pt]
O^*=&T^*\cup (\sqrt{\frac{1}{2}}+\sqrt{\frac{1}{2}}j) T^* &  \\[5pt]
I^*=&\bigcup\limits_{r=0}^4 \left(\frac{1}{2}\tau^{-1}+\frac{1}{2}\tau j+\frac{1}{2} k\right)^r T^* \quad (\text{where } \tau=\frac{\sqrt{5}+1)}{2}) &  \\[5pt]
\end{array}
$$

The group $C_n$ is cyclic of order $n$, and contains the center $-1$ if and only if $n$ is even.
The group $D_{2n}^*$ is a generalized quaternion group of order $2n$. The group $D_{2n}^*$ is called also binary dihedral and it is a central extension of the dihedral group by a group of order 2. Observe that for $n=2$, one has $D_{4}^ *=\{\pm 1,\pm j\}$, which is conjugate to $C_4=\{\pm 1,\pm i\}$. For this reason, the groups $D_{2n}^*$ are considered with indices $n\geq 3$. The case $n=3$ is also a well-known group, in fact $D_8^ *=\{\pm 1,\pm i,\pm j,\pm k\}$ is also called quaternion group. The groups $T^*$, $O^*$ and $I^*$ are central extensions of the  tetrahedral, octahedral and icosahedral group, respectively, by a group of order two; they are called  binary tetrahedral, octahedral and icosahedral, respectively.

Using  Lemma~\ref{classificationS3}, one can thus obtain a classification (up to conjugation) of the finite subgroups  of $S^3\times S^3$ which contain $\mathrm{Ker}(\Phi)$, in terms of  the finite subgroups of $S^3$. We report the classification in  Table~\ref{subgroup}. 
%Some remarks are necessary:

%\medskip

%1. 
For most cases  the group is completely determined up to conjugacy  by the first four data in the 5-tuple  $(L,L_K,R,R_K,\phi)$ and any possible  isomorphism $\phi$  gives  the same group up to  conjugacy. So  we use   Du Val's notation where the group $(L,L_K,R,R_K,\phi)$ is denoted by $(L/L_K, R/R_K)$, using a subscript  only when the isomorphism has to be specified.   
This is the case for Families $1,\, 1^\prime,\,11,\, 11^\prime,\, 26^\prime,\, 26^{\prime\prime},\,31,\, 31^\prime,\,32,\, 32^\prime,\,33$ and  $33^\prime$. 

Recalling that $\phi$ is an isomorphism  from $L/L_K$ to $R/R_K$, in the group  $(C_{2mr}/C_{2m},C_{2nr}/C_{2n})_s$ the isomorphism is
 $$\phi_s:(\cos(\pi/mr)+i \sin(\pi/mr))C_{2m}\mapsto(\cos(s\pi/nr)+i \sin(s\pi/nr))C_{2n}\,.$$ In the group $(C_{mr}/C_{m},C_{nr}/C_{n})_s$ the situation is similar and the isomorphism is $$\phi_s:(\cos(2\pi/mr)+i \sin(2\pi/mr))C_{m}\mapsto(\cos(2s\pi/nr)+i \sin(2s\pi/nr))C_{n}\,.$$

For Families 11 and $11^\prime$ we extend the isomorphisms $\phi_s$ to dihedral or binary dihedral groups sending simply $j$ to $j$.
If $L=D^*_{4mr}$, $R=D^*_{4nr}$, $L_K=C_{2m}$ and $R_K=C_{2n}$, then these isomorphisms cover all the possible cases except when $r=2$ and $m,n>1$. 
In this case we have to consider another isomorphism $f:D^*_{4mr}/C_{2m}\rightarrow D^*_{4nr}/C_{2n}$ such that: 

$$ f:\begin{cases} (\cos(\pi/2m)+i \sin(\pi/2m) )C_ {2m}\mapsto jC_{2n} \\
j C_ {2m}\mapsto(\cos(\pi/2n)+i \sin(\pi/2n))C_{2n}\end{cases}\,. $$

This is due to the fact that, if $r>2$, the  quotients $L/L_K$ and $R/R_K$ are isomorphic to a dihedral group of order greater then four where  the index two cyclic subgroup is characteristic, while if $r=2$ the quotients are dihedral groups of order four and extra isomorphisms appear.
The isomorphism $f$ gives another class of groups (the number 33 in our 
list), this family is one of the missing case in Du Val's list. However, when $m=2$ or $n=2$, one has $L=D_8^ *$ (or $R=D_8^ *$), and it is possible to conjugate $jC_2=\{\pm j\}$ to $iC_2=\{\pm i\}$ in $S^ 3$ (for instance by means of $(i+j)/\sqrt{2}$). Therefore for $m=1$ or $n=1$, the isomorphism $f$ is equivalent to the trivial isomorphism.

In Family $11^{\prime}$ the behaviour is similar. In fact if $r>2$ the isomorphisms $\phi_s$ give all the possible groups up to conjugacy, if $r=2$ and $m,n>1$ the quotients are quaternion groups of order 8 and a further  family has to be considered. This is the second missing case in \cite{duval} and Family $33^{\prime}$ in  our list where  $f$ is the following isomorphism:

$$ f:\begin{cases} (\cos(\pi/m)+i \sin(\pi/m) )C_ {m}\mapsto jC_{n} \\
j C_ {m}\mapsto(\cos(\pi/n)+i \sin(\pi/n))C_{n}\end{cases}\,. $$

The third family of groups not in  Du Val's list is Family 34 in Table~\ref{subgroup}. Note that $D^*_{4n}/C_{n}$ is cyclic of order 4 if and only if $n$ is odd. If $m$ is even while $n$ is odd, then $(C_{4m}/C_m,D^*_{4n}/C_n)$ does not contain  the kernel of $\Phi$, but if $m$ is odd, a new family appears.

The other groups in the list defined by a  non trivial automorphism between $L/L_K$ and $R/R_K$  are the groups $26''$, $32$ and $32'.$
In the first case $f$ is the identity on the subgroup $T^*$ and maps $x$ to $-x$ in the complement $O^*\setminus L^*.$ For the group $32$ (resp. $32'$) the automorphism $f$ can be chosen between the automorphism of $I^*/C_2$ (resp. $I^*$) that are not inner (see \cite[page 124]{dunbar}), in particular we choose $f$ of order two; this choice  turns out to be useful when we compute the full isometry group in Subsection \ref{o.r. isometries}.

Finally we remark that the groups  $(L,L_K,R,R_K,\phi)$ and $(R,R_K,L,L_K,\phi^{-1})$ are not conjugate unless $L$ and $R$ are conjugate in $S^3$, so the corresponding groups in $\SO(4)$ are in general not conjugate in $\SO(4)$. 
If we consider conjugation in  $\O(4)$ the situation changes, because the orientation-reversing isometry of $S^3$, sending each quaternion $z_1+z_2 j$ to its inverse $\overline{z_1}-z_2j$, conjugates the two subgroups of $\SO(4)$ corresponding to $(L,L_K,R,R_K,\phi)$ and $(R,R_K,L,L_K,\phi^{-1})$. For this reason, in Table~\ref{subgroup} only one family between  $(L,L_K,R,R_K,\phi)$ and $(R,R_K,L,L_K,\phi^{-1})$ is listed.

%\medskip

%2. The third family of groups not in  Du Val's list is Family 34 in Table~\ref{subgroup}. In this case Du Val seems not notice that the group   $C_{4m}$ and $D^*_{4n}$ have isomorphic quotients, if $m$ and $n$ are odd. This family appears in Table 4.1 of \cite{CS} (the $18^{th}$ family).

%\medskip

%3. By Proposition~\ref{classificationS3}   the groups $(L,L_K,R,R_K,\phi)$ and $(R,R_K,L,L_K,\phi^{-1})$ are not conjugated unless $L$ and
%$R$ are conjugated in $S^3$, then the corresponding groups in $\SO4$ are in general not conjugated in $\SO4$. 
%If we consider conjugation in  $\O4$ the situation changes, because the orientation-reversing isometry of $S^3$, sending each quaternion $z_1+z_2 j$ to its inverse $\overline{z_1}-z_2j$, conjugates the two subgroups of $\SO4$ corresponding to $(L,L_K,R,R_K,\phi)$ and $(R,R_K,L,L_K,\phi^{-1})$. For this reason in Table~\ref{subgroup} only one group between  $(L,L_K,R,R_K,\phi)$ and $(R,R_K,L,L_K,\phi^{-1})$ is listed.

\begin{table}
\begin{tabular}{|l|c|c|c|}%\label{tabella-gruppi}
\hline
 & $\tilde G$ & order of $G$ & \\
\hline
 1. & $(C_{2mr}/C_{2m},C_{2nr}/C_{2n})_s$ & $2mnr$ & $\gcd(s,r)=1$ \\  
 $1^{\prime}$. & $(C_{mr}/C_{m},C_{nr}/C_{n})_s$ & $(mnr)/2$ & $\gcd(s,r)=1$ $\gcd(2,n)=1$ \\
&&& $\gcd(2,m)=1$  $\gcd(2,r)=2$\\
 2. & $(C_{2m}/C_{2m},D^*_{4n}/D^*_{4n})$ & $4mn$ &  \\ 
 3. & $(C_{4m}/C_{2m},D^*_{4n}/C_{2n})$ & $4mn$ &   \\ 
 4. & $(C_{4m}/C_{2m},D^*_{8n}/D^*_{4n})$ & $8mn$ &   \\ 
 5. & $(C_{2m}/C_{2m},T^*/T^*)$ & $24m$ &   \\
 6. & $(C_{6m}/C_{2m},T^*/D^*_{8})$ & $24m$ &   \\ 
 7. & $(C_{2m}/C_{2m},O^*/O^*)$ & $48m$ &   \\
 8. & $(C_{4m}/C_{2m},O^*/T^*)$ & $48m$ &   \\ 
 9. & $(C_{2m}/C_{2m},I^*/I^*)$ & $120m$ &   \\ 
 10. & $(D^*_{4m}/D^*_{4m},D^*_{4n}/D^*_{4n})$ & $8mn$ &   \\
 11. & $(D^*_{4mr}/C_{2m},D^*_{4nr}/C_{2n})_s$ & $4mnr$ & $\gcd(s,r)=1$ \\ 
 $11^{\prime}$. & $(D^*_{2mr}/C_{m},D^*_{2nr}/C_{n})_s$ & $mnr$ &  $\gcd(s,r)=1$ $\gcd(2,n)=1$  \\
&&&  $\gcd(2,m)=1$ $\gcd(2,r)=2$\\
 12. &  $(D^*_{8m}/D^*_{4m},D^*_{8n}/D^*_{4n})$ & $16mn$ & \\
 13. &  $(D^*_{8m}/D^*_{4m},D^*_{4n}/C_{2n})$ & $8mn$ & \\
 14. &  $(D^*_{4m}/D^*_{4m},T^*/T^*)$ & $48m$ & \\
 15. &  $(D^*_{4m}/D^*_{4m},O^*/O^*)$ & $96m$ & \\
16. &  $(D^*_{4m}/C_{2m},O^*/T^*)$ & $48m$ & \\
17. &  $(D^*_{8m}/D^*_{4m},O^*/T^*)$ & $96m$ & \\
18. & $(D^*_{12m}/C_{2m},O^*/D^*_{8})$ & $48m$ & \\
19. & $(D^*_{4m}/D^*_{4m},I^*/I^*)$ & $240m$ & \\
20. & $(T^*/T^*,T^*/T^*)$ & $288$ & \\
21. & $(T^*/C_2,T^*/C_2)$ & $24$ & \\
$21^{\prime}.$ & $(T^*/C_1,T^*/C_1)$ & $12$ & \\
22. & $(T^*/D^*_{8},T^*/D^*_{8})$ & $96$ & \\
23. & $(T^*/T^*,O^*/O^*)$ & $576$ & \\
24. & $(T^*/T^*,I^*/I^*)$ & $1440$ & \\
25. & $(O^*/O^*,O^*/O^*)$ & $1152$ & \\
26. & $(O^*/C_2,O^*/C_2)$ & $48$ & \\
$26^{\prime}.$ & $(O^*/C_1,O^*/C_1)_{Id}$ & $24$ & \\
$26^{\prime\prime}.$ & $(O^*/C_1,O^*/C_1)_f$ & $24$ & \\
27. & $(O^*/D^*_{8},O^*/D^*_{8})$ & $192$ & \\
28. & $(O^*/T^*,O^*/T^*)$ & $576$ & \\
29. & $(O^*/O^*,I^*/I^*)$ & $2880$ & \\
30. &   $(I^*/I^*,I^*/I^*)$ & $7200$ & \\
31. &   $(I^*/C_2,I^*/C_2)_{Id}$ & $120$ & \\
$31^{\prime}.$ & $(I^*/C_1,I^*/C_1)_{Id}$ & $60$ & \\
32. &   $(I^*/C_2,I^*/C_2)_{f}$ & $120$ & \\
$32^{\prime}.$ & $(I^*/C_1,I^*/C_1)_f$ & $60$ & \\
33. & $(D^*_{8m}/C_{2m},D^*_{8n}/C_{2n})_f$ & $8mn$ &   $m\neq 1$  $n\neq 1$. \\ 
 $33^{\prime}$. & $(D^*_{8m}/C_{m},D^*_{8n}/C_{n})_f$ & $4mn$ &  $\gcd(2,n)=1 \gcd(2,m)=1$ \\
&&& $m\neq 1$ and $n\neq 1$.   \\
34. &  $(C_{4m}/C_{m},D^*_{4n}/C_{n})$ & $2mn$ & $\gcd(2,n)=1 \gcd(2,m)=1$ \\
\hline
\end{tabular}
\bigskip
\caption{Finite subgroups of $\SO(4)$}
\label{subgroup}
\end{table}

\subsection{Two and three-dimensional orbifolds}%\label{Seifert orbifold} 

%We give here a brief introduction to 2-orbifolds and orientable 3-orbifolds.

Roughly speaking an orbifold $\mathcal{O}$ of dimension $n$ is a paracompact Hausdorff topological space $X$ together with an atlas of open sets $(U_i,\varphi_i:\tilde U_i/\Gamma_i\to U_i)$ where $\tilde U_i$ are open subsets of $\R^n,$ $\Gamma_i$ are finite groups acting effectively on $U_i$ and $\varphi_i$ are homeomorphisms. The orbifold is smooth if  the coordinate changes $\varphi_i\circ\varphi_j^{-1}$ can be lifted to diffeomorphisms $\tilde U_i\to\tilde U_j$. There is a well-defined notion of \emph{local group} for every point $x$, namely the smallest possible group which gives a local chart for $x$, and points with trivial local group are \emph{regular points} of $\mathcal O$. Points with non-trivial local group are \emph{singular points}. The set of regular points of an orbifold is a smooth manifold. The topological space $X$ is called the \emph{underlying topological space} of the orbifold. An orbifold is \emph{orientable} if  there is an orbifold atlas such that all groups $\Gamma$ in the definition act by orientation-preserving diffeomorphisms, and the coordinate changes are lifted to orientation-preserving diffeomorphisms. For details see \cite{boileau-maillot-porti}, \cite{choi} or \cite{ratcliffe}.
%The underlying topological space of an orbifold is a manifold with boundary, and it is a manifold if the orbifold [Nnon è vero vedi Lange ArXiv 2013]

A compact orbifold is \emph{spherical} if there is an atlas as above, such that the groups $\Gamma_i$ preserve a Riemannian metric $\tilde g_i$ on $\tilde U_i$ of constant sectional curvature $1$ and the coordinate changes are lifted to isometries $(\tilde U_i,\tilde g_i)\to (\tilde U_j,\tilde g_j)$. An (orientation-preserving) diffeomorphism (resp. isometry) between spherical orbifolds $\mathcal O,\mathcal O'$ is an homeomorphism of the underlying topological spaces which can be locally lifted to an (orientation-preserving) diffeomorphism (resp. isometry) $\tilde U_i'\to\tilde U_j'$. It is known that any compact spherical orbifold  can be seen as a global quotient of $S^n$, i.e. if $\mathcal O$ is a compact spherical orbifold of dimension $n$, then  there exists a finite group of isometries of $S^n$ such that $\mathcal{O}$ is isometric to $S^n/G$ (see \cite[Theorem 13.3.10]{ratcliffe}); if the spherical orbifold is orientable $G$ is a subgroup of $\SO(n+1).$ By a result of de Rham \cite{derham} two diffeomorphic spherical orbifolds are isometric.

%If there exists an isometry between two spherical orbifolds $S^3/G$ and $S^3/G'$, then $G$ and $G'$ are conjugate in $\O(n+1)$; if the isometry is orientation-preserving the two groups are conjugate in  $\SO(n+1).$ 

Let us now explain the local models of $2$-orbifolds. The underlying topological space of a 2-orbifold is a 2-manifold with boundary. 
If $x$ is a singular point,   a neighborhood of $x$  is modelled by $D^2/\Gamma$ where the local group $\Gamma$ can be a cyclic group of rotations ($x$ is called a cone point), a group of order 2 generated by a reflection ($x$ is a  mirror reflector) or a dihedral group  generated by an index 2 subgroup of rotations and a reflection  (in this case $x$ is called a corner reflector). The local models are presented in Figure~\ref{lm2o}, a cone point or a corner reflector is labelled by its singularity index, i.e. an integer corresponding   to the order of the subgroup of rotations in $\Gamma$. We remark that the  boundary of the underlying topological space consists of mirror reflectors and corner reflectors, and the singular set might contain  in addition some   isolated points corresponding to cone points. 
If $X$ is a 2-manifold  without boundary we denote by $X(n_1,\dots,n_k)$ the 2-orbifold with underlying topological space $X$ and with $k$  cone points  of singularity index  $n_1,\ldots,n_k$. If $X$ is a 2-manifold with  non-empty connected boundary  we denote by $X(n_1,\dots,n_k;m_1,\dots,m_h)$ the  2-orbifold with $k$ cone points of singularity index $n_1,\ldots,n_k$ and with $h$ corner reflectors of singularity index $m_1,\ldots,m_h$. %Since in our cases $h\leq 3$, then the order of the list of corner reflectors is irrelevant. The 

\begin{figure}[htbp]
\centering
\begin{minipage}[c]{.3\textwidth}
\centering
\includegraphics[width=.8\textwidth]{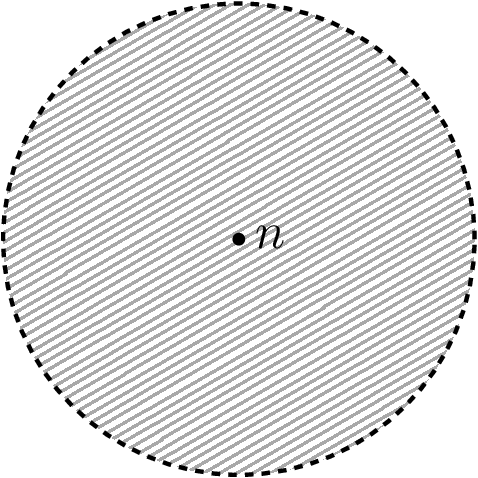} 
%\captionsetup{labelformat=empty}
\end{minipage}%
\hspace{5mm}
\begin{minipage}[c]{.3\textwidth}
\centering
\includegraphics[width=\textwidth]{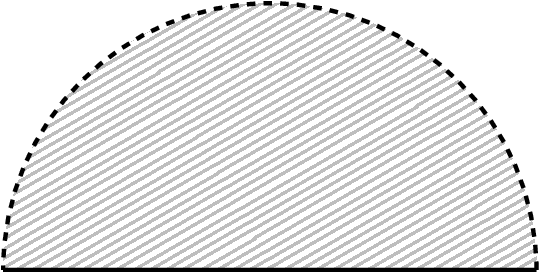} 
%\captionsetup{labelformat=empty}
\end{minipage}%
\hspace{5mm}
\begin{minipage}[c]{.3\textwidth}
\centering
\includegraphics[width=.9\textwidth]{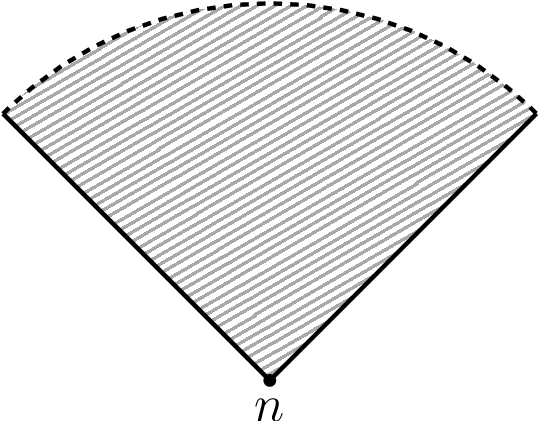} 
%\captionsetup{labelformat=empty}
\end{minipage}
\caption{Local models of 2-orbifolds. On the left, cone point. In the middle, mirror reflector. On the right, corner point.}\label{lm2o}
\end{figure}

Let us now turn the attention to 3-orbifolds. We will only consider orientable 3-orbifolds. The underlying topological space of an  orientable 3-orbifold is a 3-manifold and the singular set is a trivalent graph. 
The local models are represented in Figure~\ref{lm3o}. Excluding the vertices of the graph, the local group of  a singular point  is cyclic;  an edge of the graph is  labelled by its singularity index, that is the order of the related cyclic local groups.

\begin{figure}[htbp]
\centering
\begin{minipage}[c]{.2\textwidth}
\centering
\includegraphics[height=3cm]{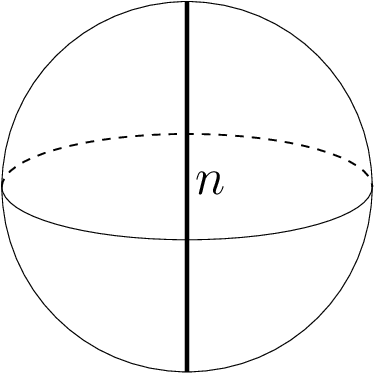} 
%\captionsetup{labelformat=empty}
\end{minipage}%
\begin{minipage}[c]{.2\textwidth}
\centering
\includegraphics[height=3cm]{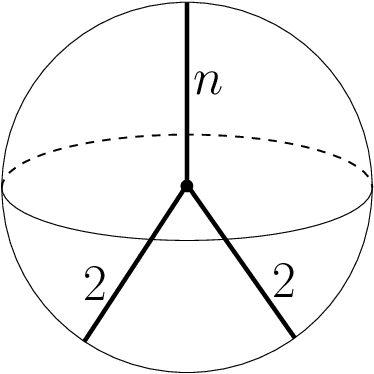} 
%\captionsetup{labelformat=empty}
\end{minipage}%
\begin{minipage}[c]{.2\textwidth}
\centering
\includegraphics[height=3cm]{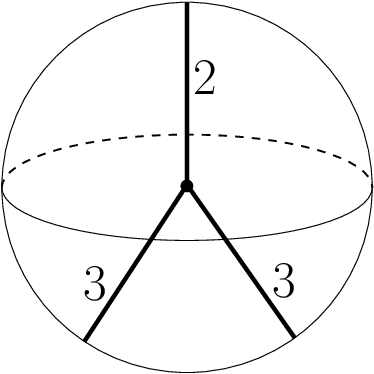} 
%\captionsetup{labelformat=empty}
\end{minipage}%
\begin{minipage}[c]{.2\textwidth}
\centering
\includegraphics[height=3cm]{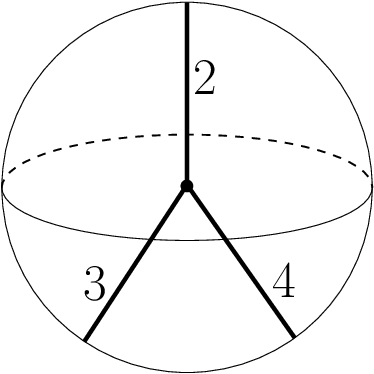} 
%\captionsetup{labelformat=empty}
\end{minipage}%
\begin{minipage}[c]{.2\textwidth}
\centering
\includegraphics[height=3cm]{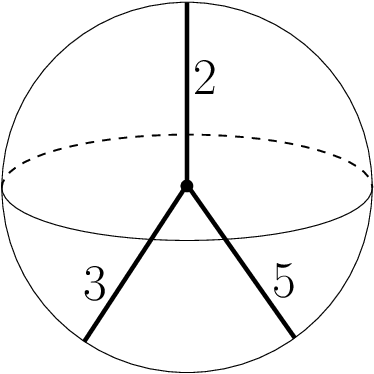} 
%\captionsetup{labelformat=empty}
\end{minipage}%
\caption{Local models of 3-orbifolds.}\label{lm3o}
\end{figure}

%\begin{figure}[htb]
%\begin{center}
%\includegraphics[height=3.0cm]{cyclic-3orbifold.eps}\hspace{0.3cm}\includegraphics[height=3.0cm]{dihedral-3orbifold.eps}\hspace{0.3cm}\includegraphics[height=3.0cm]{tetrahedral-3orbifold.eps}\hspace{0.3cm}\includegraphics[height=3.0cm]{octahedral-3orbifold.eps}\hspace{0.3cm}\includegraphics[height=3.0cm]{icosahedral-3orbifold.eps}
%\caption{Local models of 3-orbifolds}\label{lm3o}
%\end{center}
%\end{figure}

In this paper we deal with  spherical  2-orbifolds and orientable spherical 3-orbifolds, namely orbifolds $\mathcal O$ which are obtained as the quotient of $S^2$ (resp. $S^3$) by a finite group $G<\O(3)$ (resp. $G<\SO(4)$) of isometries.  An isometry between two spherical 3-orbifolds $\mathcal O=S^3/\Gamma$ and $\mathcal O'=S^3/\Gamma'$ can thus be lifted to an isometry of $S^3$ which conjugates $\Gamma$ to $\Gamma'$. If the isometry between the orbifolds is orientation-preserving, then the lift to $S^3$ is orientation-preserving. For this reason, the classification of spherical orientable 3-orbifolds $S^3/G$ up to orientation-preserving isometries corresponds to the algebraic classification of finite subgroups of $\SO(4)$ up to conjugation in $\SO(4)$.

%----------------------------------SECTION ISOMETRY-----------------------------------

\section{Isometry groups of spherical three-orbifolds} \label{sec isometry group}

The purpose of this section is to determine the isometry group of the spherical three-orbifolds $S^3/G$, once the finite subgroup $G<\SO(4)$ in the list of Table \ref{subgroup} is given. 

\subsection{Orientation-preserving isometries}\label{o-p i}
We start by determining the index 2 subgroup of orientation-preserving isometries. By the same argument as the last paragraph of Section \ref{sec spherical 3-orbifolds}, the subgroup of orientation-preserving isometries of $S^3/G$, which we denote by $\Isompr(S^3/G)$, is isomorphic to $\Norm_{\SO(4)}(G)/G$. The latter is in turn isomorphic to the quotient $\Norm_{S^3\times S^3}(\tilde G)/\tilde G$.

A special case of Lemma \ref{classificationS3} is the following:

%VANNO MESSI PRIMA

%CAPITOLO PRECEDENTE:
\begin{Lemma}\label{lemma diagramma}
Let $\tilde G=(L,L_K,R,R_K,\phi)$ be a finite subgroup of $S^3\times S^3 $ containing $\mathrm{Ker}(\Phi)$. An element $(g,f)\in S^3\times S^3 $ is contained in the normalizer $N_{S^3\times S^3}(G)$ if and only if the following three conditions are satisfied:
\begin{enumerate}
\item $(g,f)\in  N_{S^3}(L)\times N_{S^3}(R)$;
\item $(g,f)\in  N_{S^3}(L_K)\times N_{S^3}(R_K)$;
\item the following diagram commutes:

\begin{equation} \label{diagramma quadrato}
\begin{gathered}
\xymatrix{
L/L_{K} \ar[d]^-{\alpha}\ar[r]^-{\phi} & R/R_{K} \ar[d]^-{\beta}\\
L/L_{K}   \ar[r]^-{\phi} & R/R_{K} \\
}
\end{gathered}
\end{equation}

where $\alpha(xL_K)=g^{-1}xgL_K$ and   $\beta(yR_K)=f^{-1}yfR_K$. 
\end{enumerate}
   
\end{Lemma}

First, it is necessary to understand the normalizers of the finite subgroups of $S^3$. These are listed for instance in \cite{mccullough}. We report a list here:

\begin{equation} \label{list normalizers}
\begin{array}{rll} 
\Norm_{S^3}(C_n)=& \O(2)^* & \quad \text{if }n>2 \\
\Norm_{S^3}(C_2)=& S^3 & \quad\text{if }n=2 \\
\Norm_{S^3}(D^*_{4n})=& D^*_{8n} & \quad\text{if }n>2 \\
\Norm_{S^3}(D^*_{8})=& O^* & \quad\text{if }n=2 \\
\Norm_{S^3}(T^*)=& O^* & \\
\Norm_{S^3}(O^*)=& O^* & \\
\Norm_{S^3}(I^*)=& I^* & \\
\end{array}
\end{equation}

We split the computation in several cases, including in each case those groups for which the techniques involved are comparable. Families 1, $1'$, 11, $11'$ are treated in a systematic way in Case \ref{casedifficult}, although for some special values of the indices, they should in principle fall in the categories of some of the previous cases. In the following we denote by $D_n$ the dihedral group of order $n$ and $O$ is the octahedral group (that is isomorphic to $O^*/C_2$, and also to the symmetric group on 4 elements). If $A$ is an abelian group, we denote by $\mathrm{Dih}(A)$ the semidirect product $\mathbb{Z}_2\ltimes A$ where the involution in $\mathbb{Z}_2$ inverts by conjugation each element of the normal subgroup $A.$

\begin{case}
$\tilde G$ is a product, i.e. $L=L_K$ and $R=R_K$.
\end{case}

In this case $L/L_K$ and $R/R_K$ are trivial groups, hence the conditions of Lemma \ref{lemma diagramma} are trivially satisfied. Therefore $\Norm_{S^3\times S^3}(\tilde G)=\Norm_{S^3}(L)\times\Norm_{S^3}(R)$. Hence $\Isom\!^+\!(S^3/G)=(\Norm_{S^3}(L)/L)\times(\Norm_{S^3}(R)/R)$. This is the case of Families 2, 5, 7, 9, 10, 14, 15, 19 and Groups 20, 23, 24, 25, 29, 30 of Table \ref{subgroup}.

\begin{case}
$L/L_K\cong R/R_K\cong\Z_2$ and $L_K,R_K$ are not generalized quaternion groups.
\end{case}

It turns out that, for the pairs $(L,L_K)=(C_{4n},C_{2n}),(D^*_{4n},C_{2n}),(O^*,T^*)$, the normalizer $\Norm_{S^3}(L)$ also normalizes $L_K$ (or analogously for $R$ and $R_K$). In this case, the condition of commutativity of the diagram \eqref{diagramma quadrato} is trivially satisfied, since the identity is the only automorphism of $\Z_2$. This shows that, for Families 3, 8, 16 and Group 28 one has $\Norm_{S^3\times S^3}(\tilde G)=\Norm_{S^3}(L)\times\Norm_{S^3}(R)$.

To understand the isomorphism type of the orientation-preserving isometry group, for instance for the group in Family 3, namely 
$$\tilde G=(C_{4m}/C_{2m},D^*_{4n}/C_{2n})\,,$$
we consider
$$\Isompr(S^3/G)=\Norm_{S^3\times S^3}(\tilde G)/\tilde G=\left((\O(2)^*\times D_{8n}^*)/(C_{2m}\times C_{2n})\right)/(\tilde G/(C_{2m}\times C_{2n}))\,.$$
Therefore we get
$$\Isompr(S^3/G)\cong(\O(2)\times D_4)/\Z_2$$
and, observing that $D_4\cong \Z_2\times\Z_2$,
%, and keeping track of the action of $\Z_2$ on $\O(2)\times D_4$
one concludes that
$$\Isompr(S^3/G)\cong\O(2)\times \Z_2\,.$$

\begin{case}
$L/L_K\cong R/R_K\cong\Z_2$ and $L_K$ or $R_K$ is a generalized quaternion group.
\end{case}

When $(L,L_K)=(D^*_{8n},D^*_{4n})$, $D^*_{4n}$ is not normal in $\Norm_{S^3}(L)=D^*_{16n}$, and the normalizer of $D^*_{4n}$ is $D^*_{8n}$. This is the case of Families 4, 12, 13 and 17. Basically, here one has to consider
$$\Norm_{S^3\times S^3}(\tilde G)=(\Norm_{S^3}(L)\cap\Norm_{S^3}(L_K))\times (\Norm_{S^3}(R)\cap\Norm_{S^3}(R_K))$$
and apply the above strategy to successively compute $\Isom(S^3/G)$. For instance, for Family 13,
$$\tilde G=(D^*_{8m}/D^*_{4m},D^*_{4n}/C_{2n})\,,$$
we have
$$\Norm_{S^3\times S^3}(\tilde G)=D^*_{8m}\times D^*_{8n}\,.$$
Hence one gets
$$\Isompr(S^3/G)\cong(\Z_2\times D_4)/\Z_2\cong\Z_2\times\Z_2\,.$$

\begin{case} \label{caseZ3}
$L/L_K\cong R/R_K\cong\Z_3$.
\end{case}
This case includes Family 6 and Group 22. For Group 22, namely
$$\tilde G=(T^*/D^*_{8},T^*/D^*_{8})\,,$$
we have $\Norm_{S^3}(L)=\Norm_{S^3}(R)=O^*$, and $L_K=R_K=D_8^*$ is normal in $O^*$. The induced action of the elements of $T^*$ on $T^*/D_8^*=\Z_3$ is the identity, hence the normalizer of $\tilde G$ contains $T^*\times T^*$. Moreover, the induced action of elements of $O^*\setminus T^*$ on $\Z_3$ is dihedral (i.e. sends each element of $\Z_3$ to its inverse), hence elements of $O^*\setminus T^*$ on the left side have to be paired to elements of $O^*\setminus T^*$ on the right side to make the diagram \eqref{diagramma quadrato} commutative. Hence
$$\Norm_{S^3\times S^3}(\tilde G)=(O^*/T^*,O^*/T^*)$$
and the isometry group is $D_6$. By a similar argument, one checks that the normalizer of 
$$\tilde G=(C_{6m}/C_{2m},T^*/D^*_{8})$$
is 
$$\Norm_{S^3\times S^3}(\tilde G)=(\O(2)^*/S^1,O^*/T^*)\,.$$
Finally, one can compute
\begin{align*}
\Isompr(S^3/G)&=\left((\O(2)^*/S^1,O^*/T^*)/(C_{2m}\times D_8^*)\right)/((C_{6m}/C_{2m},T^*/D^*_{8})/(C_{2m}\times D_8^*))\\
&\cong\mathrm{Dih}(S^1\times \Z_3)/\Z_3\cong \O(2)\,.
\end{align*}

\begin{case}
$L/L_K\cong R/R_K\cong D_6$.
\end{case}
For Group 27, in a very similar fashion as Group 22, the normalizer in $S^3\times S^3$ is $(O^*/T^*,O^*/T^*)$ and thus the orientation-preserving isometry group is isomorphic to $\Z_3$. The other groups to be considered here are those in Family 18, namely
$$\tilde G=(D^*_{12m}/C_{2m},O^*/D^*_{8})\,.$$
Again, arguing similarly to Case \ref{caseZ3}, one sees that
$$\Norm_{S^3\times S^3}(\tilde G)=(D_{24m}^*/C_{4m},O^*/D_8^*)\,.$$
Therefore the orientation-preserving  isometry group is $\Z_2$. %(serve dire qualcosa di piu? mi sembra che sia vero, infatti il quoziente ha ordine 6 e deve contenere qualcosa di diedrale, quindi rimane solo $D_6$.)

\begin{table}
\begin{tabular}{|l|c|c|c|c|}%\label{tabella-gruppi}
\hline
 &  $\tilde G$ & $\Isompr(S^3/G)$ & $\Isom_0(S^3/G)$ & $\pi_0 \Isompr(S^3/G)$ \\
\hline
 1. & $(C_{2mr}/C_{2m},C_{2nr}/C_{2n})_s$ & $\mathrm{Dih}(S^1\times S^1)$  & $S^1\times S^1$ & $\Z_ 2$   \\  
 $1^{\prime}$. & $(C_{mr}/C_{m},C_{nr}/C_{n})_s$ & $\mathrm{Dih}(S^1\times S^1)$ & $S^1\times S^1$ & $\Z_ 2$  \\
 2. & $(C_{2m}/C_{2m},D^*_{4n}/D^*_{4n})$ & $\O(2)\times \Z_2$ & $S^1$ & $\Z_2\times\Z_ 2$  \\ 
 3. & $(C_{4m}/C_{2m},D^*_{4n}/C_{2n})$ & $\O(2)\times \Z_2$ & $S^1$ & $\Z_2\times\Z_ 2$    \\ 
 4. & $(C_{4m}/C_{2m},D^*_{8n}/D^*_{4n})$ & $\O(2)$ & $S^1$ & $\Z_ 2$   \\ 
 5. & $(C_{2m}/C_{2m},T^*/T^*)$ & $\O(2)\times\Z_2$ & $S^1$ & $\Z_2\times\Z_ 2$   \\
 6. & $(C_{6m}/C_{2m},T^*/D^*_{8})$ & $\O(2)$ & $S^1$ & $\Z_ 2$   \\ 
 7. & $(C_{2m}/C_{2m},O^*/O^*)$ & $\O(2)$ & $S^1$ & $\Z_ 2$   \\
 8. & $(C_{4m}/C_{2m},O^*/T^*)$ & $\O(2)$ & $S^1$ & $\Z_ 2$   \\ 
 9. & $(C_{2m}/C_{2m},I^*/I^*)$ & $\O(2)$ & $S^1$ & $\Z_ 2$    \\ 
 10. & $(D^*_{4m}/D^*_{4m},D^*_{4n}/D^*_{4n})$ & $\Z_2\times\Z_2$ & $\{1\}$ & $\Z_2\times\Z_2$  \\
 11. & $(D^*_{4mr}/C_{2m},D^*_{4nr}/C_{2n})_s$ & $\Z_2$ & $\{1\}$ & $\Z_2$  \\ 
 $11^{\prime}$. & $(D^*_{2mr}/C_{m},D^*_{2nr}/C_{n})_s$ & $\Z_2\times\Z_2$ & $\{1\}$ & $\Z_2\times\Z_2$    \\
 12. &  $(D^*_{8m}/D^*_{4m},D^*_{8n}/D^*_{4n})$ & $\Z_2$ & $\{1\}$ & $\Z_2$ \\
 13. &  $(D^*_{8m}/D^*_{4m},D^*_{4n}/C_{2n})$ & $\Z_2\times\Z_2$ & $\{1\}$ & $\Z_2\times\Z_2$ \\
 14. &  $(D^*_{4m}/D^*_{4m},T^*/T^*)$ & $\Z_2\times\Z_2$ & $\{1\}$ & $\Z_2\times\Z_2$ \\
 15. &  $(D^*_{4m}/D^*_{4m},O^*/O^*)$ & $\Z_2$ & $\{1\}$ & $\Z_2$ \\
16. &  $(D^*_{4m}/C_{2m},O^*/T^*)$ & $\Z_2\times\Z_2$ & $\{1\}$ & $\Z_2\times\Z_2$ \\
17. &  $(D^*_{8m}/D^*_{4m},O^*/T^*)$ & $\Z_2$ & $\{1\}$ & $\Z_2$ \\
18. & $(D^*_{12m}/C_{2m},O^*/D^*_{8})$ & $\Z_2$ & $\{1\}$ & $\Z_2$ \\
19. & $(D^*_{4m}/D^*_{4m},I^*/I^*)$ & $\Z_2$ & $\{1\}$ & $\Z_2$ \\
20. & $(T^*/T^*,T^*/T^*)$ & $\Z_2\times\Z_2$ & $\{1\}$ & $\Z_2\times\Z_2$ \\
21. & $(T^*/C_2,T^*/C_2)$ & $\Z_2$ & $\{1\}$ & $\Z_2$ \\
$21^{\prime}.$ & $(T^*/C_1,T^*/C_1)$ & $\Z_2\times\Z_2$ & $\{1\}$ & $\Z_2\times\Z_2$ \\
22. & $(T^*/D^*_{8},T^*/D^*_{8})$ & $D_6$ & $\{1\}$ & $D_6$ \\
23. & $(T^*/T^*,O^*/O^*)$ & $\Z_2$ & $\{1\}$ & $\Z_2$ \\
24. & $(T^*/T^*,I^*/I^*)$ & $\Z_2$ & $\{1\}$ & $\Z_2$  \\
25. & $(O^*/O^*,O^*/O^*)$ & $\{1\}$ & $\{1\}$ & $\{1\}$ \\
26. & $(O^*/C_2,O^*/C_2)$ & $\{1\}$ & $\{1\}$ & $\{1\}$ \\
$26^{\prime}.$ & $(O^*/C_1,O^*/C_1)_{Id}$ & $\Z_2$ & $\{1\}$ & $\Z_2$ \\
$26^{\prime\prime}.$ & $(O^*/C_1,O^*/C_1)_f$ & $\Z_2$ & $\{1\}$ & $\Z_2$ \\
27. & $(O^*/D^*_{8},O^*/D^*_{8})$ & $\Z_3$ & $\{1\}$ & $\Z_3$ \\
28. & $(O^*/T^*,O^*/T^*)$ & $\Z_2$ & $\{1\}$ & $\Z_2$ \\
29. & $(O^*/O^*,I^*/I^*)$ & $\{1\}$ & $\{1\}$ & $\{1\}$ \\
30. &   $(I^*/I^*,I^*/I^*)$ & $\{1\}$ & $\{1\}$ & $\{1\}$ \\
31. &   $(I^*/C_2,I^*/C_2)_{Id}$ & $\{1\}$ & $\{1\}$ & $\{1\}$ \\
$31^{\prime}.$ & $(I^*/C_1,I^*/C_1)_{Id}$ & $\Z_2$ & $\{1\}$ & $\Z_2$ \\
32. &   $(I^*/C_2,I^*/C_2)_{f}$ & $\{1\}$ & $\{1\}$ & $\{1\}$ \\
$32^{\prime}.$ & $(I^*/C_1,I^*/C_1)_f$ & $\Z_2$ & $\{1\}$ & $\Z_2$ \\
33. & $(D^*_{8m}/C_{2m},D^*_{8n}/C_{2n})_f$ &$\{1\}$ & $\{1\}$ & $\{1\}$    \\ 
 $33^{\prime}$. & $(D^*_{8m}/C_{m},D^*_{8n}/C_{n})_f$ & $\Z_2$ & $\{1\}$ & $\Z_2$   \\
34. &  $(C_{4m}/C_{m},D^*_{4n}/C_{n})$ & $\O(2)$ & $S^1$ & $\Z_ 2$  \\
\hline
\end{tabular}
\bigskip
\caption{Table of orientation-preserving isometry groups for $C_{2m}\neq C_2$, $D_{4m}^*\neq D_4^*,D_8^*$, $r >2$}
\label{tableisometry}
\end{table}

\begin{case}
$L/L_K\cong R/R_K\cong T,T^*,O,O^*,I$ or $I^*$.
\end{case}
We start by considering Group 21, i.e. $\tilde G=(T^*/C_2,T^*/C_2)$. We have $\Norm_{S^3}(L)=\Norm_{S^3}(R)=O^*$. On the other hand, %since $O^*/C_2=O$ has trivial centre, 
observing that the centre of $O^*$ is precisely $C_2$, the normalizer of $\tilde G$ turns out to be $(O^*/C_2,O^*/C_2)$. In fact, in order to satisfy Lemma \ref{lemma diagramma}, any element of the form $(g,1)$ which normalizes $\tilde G$ must have $g\in C_2$. Hence the orientation-preserving  isometry group is $\Z_2$. 

By analogous considerations, since $\Norm_{S^3} (O^*)=O^*$ and $\Norm_{S^3} (I^*)=I^*$, the orientation-preserving  isometry groups for Groups 26 and 31 are trivial. Group 32, namely $\tilde G=(I^*/C_2,I^*/C_2)_{f}$, is defined  by means of a non-inner automorphism of $I^*$. However, since any  automorphism preserves the centre and $O^*/C_2$  has trivial centre, the same argument applies and the orientation-preserving  isometry group is trivial.

For Group $21'$, namely $\tilde G=(T^*/C_1,T^*/C_1)$, the normalizer is again $(O^*/C_2,O^*/C_2)$, since the centre of $O^*$ is $C_2$. The orientation-preserving isometry group is $\Z_2\times\Z_2$. Finally, the normalizer of Groups $26'$ and $26''$ is Group 26, the normalizer of Group $31'$ is Group 31, and the normalizer of Group $32'$ is Group 32. Therefore Groups $26'$, $26''$, $31'$ and $32'$ have orientation-preserving  isometry groups isomorphic to $\Z_2$.

\begin{case}
$L/L_K\cong R/R_K\cong \Z_4,\Z_2\times\Z_2$ or $D_8^*$.
\end{case}
This is the case of Families 33, $33'$ and 34. %(we exclude the occurrences concerning Families 1,1',11,11', which will be discussed later).
The definition of Family 33, namely
$$\tilde G=(D^*_{8m}/C_{2m},D^*_{8n}/C_{2n})_f\,,$$
makes use of a non-trivial automorphism of $\Z_2\times\Z_2$. One can easily check that in $\Norm_{S^3}(L)\times \Norm_{S^3}(R)=D^*_{16m}\times D^*_{16m}$, it is not possible to obtain a pair $(g,f)$ which makes the diagram \eqref{diagramma quadrato} commute, unless $(g,f)$ is already in $\tilde G$. This is essentially due to the fact that the non-trivial automorphism of $D_4\cong \Z_2\times\Z_2$ cannot be extended to an automorphism of $D_8$. Hence the orientation-preserving isometry group is trivial. Clearly any group in Family 33' is normalized by the corresponding group in Family 33, and thus the group of orientation-preserving isometries is $\Z_2$.

Family 34, namely $$\tilde G=(C_{4m}/C_{m},D^*_{4n}/C_{n})\,,$$
is defined by means of the fact that the quotients $L/L_K$ and $R/R_K$ are both isomorphic to $\Z_4$. The normalizer here is $(\O(2)^*/S^1,D_{8n}^*/D_{4n}^*)$. The orientation-preserving isometry group is a dihedral extension of 
$(S^1\times D_{4n}^*)/\tilde G$. The latter is a quotient of $S^1\times D_4^*\cong S^1\times \Z_4$ by a diagonal action of $\Z_4$, and therefore still isomorphic to $S^1$. In conclusion, the group of orientation-preserving isometries is isomorphic to $\O(2)$.

\begin{case}
Exceptional cases for small $m$ or $n$.
\end{case}
It is necessary to distinguish from the results obtained above some special cases for small values of the indices $m$ and $n$. Indeed, as in the list \eqref{list normalizers}, when $n=2$ the subgroup $C_2$ of $S^3$ is normalized by the whole $S^3$, whereas the normalizer of $D_8^*$ is $O^*$. Therefore one obtains different isometry groups when $L,R=C_2$ or $L,R=D_8^*$. The results are collected in Table \ref{tableisometryexceptions}.

Excluding Families 1,$1'$,11, $11'$ which are discussed below, to compute the isometry group in these special cases one uses the same approach as above. For Families 2, 5, 7, 9, 10, 14, 15, 19, which are products, one obtains straightforwardly the result. 

For the other groups in which $D^*_8$ may appear, one checks that there are no difference with the case $m>1$ since $O^*$, which is the normalizer of $D^*_8$, does not normalize any subgroup of order 4 of $D^*_8$.

% In the situation $L/L_K=D_8^*/D_4^*$ (or $R/R_K=D_8^*/D_4^*$), which is actually equivalent to considering $D_8^*/C_4$ since $D_4^*=\{\pm 1,\pm j\}$ and $C_4=\{\pm 1,\pm i\}$, $\Norm_{S^3}(D_8^*)=O^*$ does not preserve the subgroup $C_4$. Hence, as above, one must consider $\Norm_{S^3}(D_8^*)\cap \Norm_{S^3}(C_4)=D_{16}^*$. [In questo contesto è rindondante] 

%For the groups in Family 3, namely $\tilde G=(C_{4m}/C_{2m},D^*_{4n}/C_{2n})$, when $n=2$ there is no difference in the result with respect to the general case $n>2$.

%Another interesting example arises from Family 13. In fact, for $\tilde G=(D^*_{8m}/D^*_{4m},D^*_{4n}/C_{2n})$, if $m>1$ and $n=1$ the normalizer is $D_{8m}^*\times D_{16}^* $ and thus the orientation-preserving isometry group is $\Z_2\times \Z_2$. 
%On the other hand, if $m=n=1$ then the normalizer is $D_{16}^*\times D_{16}^* $ and thus the orientation-preserving isometry group is $\Z_2\times \Z_2\times \Z_2$. A different, although similar, situation comes from the case $m=1$, $n>2$ in the group $\tilde G=(D^*_{8m}/D^*_{4m},D^*_{4n}/C_{2n})$. Here the normalizer is $D_{16}^*\times D_{8n}^*$ and $\Isompr(S^3/G)$ is $\Z_2\times \Z_2\times \Z_2$. 
%The special case of Family 16 are treated similarly.

\begin{table}\label{tabella isometrie small indices}
\begin{tabular}{|l|c|c|c|c|} 
\hline
 & $\tilde G$ & $\Isompr(S^3/G)$ & $\Isom_0(S^3/G)$ & $\pi_0 \Isompr(S^3/G)$ \\
\hline
 1. & $(C_{2m}/C_{2m},C_{2n}/C_{2n})$ & $\O(2)\times \O(2)$  & $S^1\times S^1$ & $\Z_2\times \Z_2$  \\
  &  $(C_{4m}/C_{2m},C_{4n}/C_{2n})$ & $\O(2)\widetilde{\times}\O(2)$  & $S^1\times S^1$ & $\Z_2\times \Z_2$ \\
  & $(C_{2}/C_{2},C_{2n}/C_{2n})$ & $\SO(3)\times \O(2)$  & $ \SO(3)\times S^1$ & $\Z_2$  \\
  & $(C_2/C_2,C_2/C_2)$ & $\mathrm{P}\SO(4)$ & $\mathrm{P}\SO(4)$ & $\{1\}$ \\
 $1^{\prime}$. & $(C_{2m}/C_{m},C_{2n}/C_{n})$ & $\O(2)^*\widetilde{\times} \O(2)^*$ & $S^1\times S^1$ & $\Z_2\times \Z_2$ \\
 & $(C_2/C_1,C_{2n}/C_n)$ & $S^3\widetilde{\times}\O(2)^*$ & $S^3\widetilde{\times}S^1$ & $\Z_2$ \\
 & $(C_2/C_1,C_2/C_1)$ & $\SO(4)$ & $\SO(4)$ & $\{1\}$ \\
 2. & $(C_{2}/C_{2},D^*_{4n}/D^*_{4n})$ & $\SO(3)\times \Z_2$ & $\SO(3)$ & $\Z_2$  \\ 
 & $(C_{2m}/C_{2m},D^*_{8}/D^*_{8})$ & $\O(2)\times D_6$ & $S^1$ & $\Z_2\times D_6$  \\ 
 & $(C_{2}/C_{2},D^*_{8}/D^*_{8})$ & $\SO(3)\times D_6$ & $\SO(3)$ & $D_6$  \\ 
 %3.4. & $(C_{4m}/C_{2m},D^*_{8}/C_{4})$ & $\O(2)\times \Z_2$ & $S^1$ & $\Z_2\times\Z_ 2$    \\ 
 %4. & $(C_{4m}/C_{2m},D^*_{8}/C_{4})$ & $\O(2)\times \Z_2$ & $S^1$ & $\Z_2\times\Z_ 2$    \\ 
 5. & $(C_{2}/C_{2},T^*/T^*)$ & $\SO(3)\times\Z_2$ & $\SO(3)$ & $\Z_2$   \\
% 6. & $(C_{6m}/C_{2m},T^*/D^*_{8})$ & $24m$ &   \\ 
 7. & $(C_{2}/C_{2},O^*/O^*)$ & $\SO(3)$ & $\SO(3)$ & $\{1\}$   \\
 %8. & $(C_{4m}/C_{2m},O^*/T^*)$ & $\O(2)$ & $S^1$ & $\Z_ 2$   \\ 
 9. & $(C_{2}/C_{2},I^*/I^*)$ & $\SO(3)$ & $\SO(3)$ & $\{1\}$    \\ 
 10. & $(D^*_{8}/D^*_{8},D^*_{4n}/D^*_{4n})$ & $D_6\times\Z_2$ & $\{1\}$ & $D_6\times\Z_2$  \\
& $(D^*_{8}/D^*_{8},D^*_{8}/D^*_{8})$ & $D_6\times D_6$ & $\{1\}$ & $D_6\times D_6$  \\

 11. & $(D^*_{4m}/C_{2m},D^*_{4n}/C_{2n})$ & $\Z_2\times\Z_2\times\Z_2$ & $\{1\}$ & $\Z_2\times\Z_2\times\Z_2$   \\ 
 & $(D^*_{8m}/C_{2m},D^*_{8n}/C_{2n})$ & $\Z_2\times\Z_2\times\Z_2$ & $\{1\}$ & $\Z_2\times\Z_2\times\Z_2$   \\ 
  & $(D^*_{8}/C_{2},D^*_{8}/C_{2})$ & $O$ & $\{1\}$ & $O$   \\ 
 $11^{\prime}$. & $(D^*_{8}/C_{1},D^*_{8}/C_{1})$ & $D_6\times\Z_2$ & $\{1\}$ & $D_6\times\Z_2$    \\
 %12. &  $(D^*_{8m}/D^*_{4m},D^*_{8}/C_{4})$ & $\Z_2\times\Z_2$ & $\{1\}$ & $\Z_2\times\Z_2$ \\
 %&  $(D^*_{8}/C_{4},D^*_{8}/C_{4})$ & $\Z_2\times\Z_2\times\Z_2$ & $\{1\}$ & $\Z_2\times\Z_2\times\Z_2$ \\
 %13. &  $(D^*_{8}/D^*_{4},D^*_{4n}/C_{2n})$ & $\Z_2\times\Z_2\times\Z_2$ & $\{1\}$ & $\Z_2\times\Z_2\times\Z_2$ \\
 14. &  $(D^*_{8}/D^*_{8},T^*/T^*)$ & $D_6\times\Z_2$ & $\{1\}$ & $D_6\times\Z_2$ \\
 15. &  $(D^*_{8}/D^*_{8},O^*/O^*)$ & $D_6$ & $\{1\}$ & $D_6$ \\
%16. &  $(D^*_{8}/C_{4},O^*/T^*)$ & $\Z_2\times\Z_2$ & $\{1\}$ & $\Z_2\times\Z_2$ \\
%17. &  $(D^*_{8m}/D^*_{4m},O^*/T^*)$ & $\Z_2$ & $\{1\}$ & $\Z_2$ \\
%18. & $(D^*_{12m}/C_{2m},O^*/D^*_{8})$ & $48m$ & \\
19. & $(D^*_{8}/D^*_{8},I^*/I^*)$ & $D_6$ & $\{1\}$ & $D_6$ \\
%33. & $(D^*_{8m}/C_{2m},D^*_{8n}/C_{2n})_f$ & $8mn$ &    \\ 
 %$33^{\prime}$. & $(D^*_{8m}/C_{m},D^*_{8n}/C_{n})_f$ & $4mn$ &   \\
%34. &  $(C_{4m}/C_{m},D^*_{4n}/C_{n})$ & $2mn$ &  \\
\hline
\end{tabular}
\bigskip
\caption{Table of orientation-preserving isometry groups for small indices}
\label{tableisometryexceptions}
\end{table}

\begin{case} \label{casedifficult}
Families 1, 1', 11, 11'.
\end{case}
We are left with the treatment of the groups for which $L_K$ and $R_K$ are both cyclic groups, while $L$ and $R$ are both of the same type, either cyclic of generalized dihedral. For Family 1,
$$\tilde G=(C_{2mr}/C_{2m},C_{2nr}/C_{2n})_s\,,$$
where we recall that the index $s$ denotes the isomorphism $\phi:\Z_r\to\Z_r$ given by $1\mapsto s$,
 we have 
 $$\Norm_{S^3}(L)=\Norm_{S^3}(R)=\O(2)^*\,.$$
Moreover $\O(2)^*$ preserves the subgroups $L_K$ and $R_K$. We need to check the commutativity of diagram \eqref{diagramma quadrato}. For this purpose, observe that the diagram commutes trivially when we choose elements of the form $(g,1)$ or $(1,g)$, for $g\in S^1$. On the other hand, the induced action of elements of $\O(2)^*\setminus S^1$ on $\Z_r$ is dihedral, hence the normalizer is 
$$\Norm_{S^3\times S^3}(\tilde G)=(\O(2)^*/S^1,\O(2)^*/S^1)$$
unless $r=2$ (or $r=1$). The isometry group $\Norm_{S^3\times S^3}(\tilde G)/\tilde G$ is a dihedral extension of $(S^1\times S^1)/\tilde G$, the latter being again isomorphic to $S^1\times S^1$ (for an explicit isomorphism, see Equation \eqref{isomorfismo gamma} below). Hence we get
$\Isompr(S^3/\tilde G)\cong \mathrm{Dih}(S^1\times S^1)$. The same result is recovered analogously for Family $1'$, unless $r=2$.

For some special cases of Family 1, when $r=2$ or $r=1$, $\Norm_{S^3\times S^3}(\tilde G)=\O(2)^*\times \O(2)^*$. For $r=1$, one has $\Isompr(S^3/\tilde G)\cong \O(2)\times \O(2)$. For $r=2$, the orientation-preserving isometry group is $\O(2)\times\O(2)/(-1,-1)=\O(2)\widetilde{\times}\O(2)$. When $r=2$ the group of orientation-preserving isometries of Family $1'$ turns out to be instead $\O(2)^*\times\O(2)^*/(-1,-1)=\O(2)^*\widetilde{\times}\O(2)^*$. The latter cases are collected in Table \ref{tableisometryexceptions}. By means of this analysis, we also recover the trivial cases of $S^3$ itself and of projective space $S^3/\{\pm 1\}$, which correspond to the groups $(C_2/C_1,C_2/C_1)$ and $(C_2/C_2,C_2/C_2)$, having orientation-preserving isometry group $\SO(4)$ and $\mathrm{P}\SO(4)$ respectively.

For Family 11, namely
$$\tilde G=(D^*_{4mr}/C_{2m},D^*_{4nr}/C_{2n})_s\,,$$
it is not difficult to see that the normalizer is 
$$(D^*_{8mr}/C_{2m},D^*_{8nr}/C_{2n})_s$$
unless $r=1$ or $r=2$.
Hence $\Isompr(S^3/G)$ is isomorphic to $\Z_2$. 
When $r=1$ the normalizer of $(D_{4m}^*/C_{2m},D_{4n}^*/C_{2n})$ is $D_{8m}^*\times D_{8n}^*$ (also if $m$ of $n$ are equal to $2$), and the computation of the orientation-preserving isometry group follows. When $r=2$, the normalizer of  $(D_{8m}^*/C_{2m},D_{8n}^*/C_{2n})$ turns out to be $(D_{16m}^*/D_{8m}^*,D_{16n}^*/D_{8n}^*)$, also if one between $m$ or $n$ is equal to one. However, when $m=n=1$, the normalizer is $(O^*/D^*_8,O^*/D^*_8)$ and the orientation-preserving isometry group is isomorphic to $O=O^*/C_2$.

A similar argument shows that for Family $11'$, $\Isompr(S^3/G)$ is isomorphic to $\Z_2\times\Z_2$, also if $r=4$ and one among $m$ and $n$ equals $1$. Thus the only exception is for the group $(D_{8}^*/C_{1},D_{8}^*/C_{1})$.

\subsection{Orientation-reversing isometries}\label{o.r. isometries}
We now compute the full isometry groups of spherical orbifolds $S^3/G$. We collect the results of this section in Table \ref{tableisometryreversing} where we list the groups such that $\Isom(S^3/G)\neq\Isompr(S^3/G)$ and the full isometry groups of the corresponding orbifolds are described.  If the spherical orbifold does not admit any orientation-reversing isometry then $\Isom(S^3/G)$ can be deduced from Tables  \ref{tableisometry} and \ref{tableisometryexceptions}.

An orientation-reversing element of $\O(4)$ is of the form $h\rightarrow p\bar{h}q^{-1}$ where $p$ and $q$ are elements of $S^3$ and $\bar{h}$ is the conjugate element of $h$ (see \cite[p.58]{duval}).  We will denote this isometry by $\overline{\Phi}_{p,q}$. 
Let us remark that 
\begin{equation} \label{or rev conjugation rule}
\overline{\Phi}_{p,q}\overline{\Phi}_{l,r}\overline{\Phi}_{p,q}^{-1}=\overline{\Phi}_{prp^{-1},qlq^{-1}}\,.
\end{equation} 
We state two lemmata whose proofs are straightforward. 

\begin{Lemma}\label{lemma orientation-reversing primo}
Let $\tilde{G}=(L,L_K,R,R_K,\phi)$ be a finite subgroup of $S^3\times S^3 $ containing $\mathrm{Ker}(\Phi)$. If  $\overline{\Phi}_{p,q}$ normalizes $\Phi(\tilde{G})$ then the following conditions are satisfied:
\begin{enumerate}
\item $p^{-1}Lp=R$ and $q^{-1}Rq=L$;
\item  $pL_Kp^{-1}=R_K$ and $qR_Kq^{-1}=L_K.$
\end{enumerate}
   
\end{Lemma}

So if an orientation reversing isometry of $S^3$ normalizes $G$, then  we can suppose  up to conjugacy that  $L=R$ and $L_K=R_K.$

\begin{Lemma}\label{lemma orientation-reversing secondo}

Let $\tilde{G}=(R,R_K,R,R_K,\phi)$ be a finite subgroup of $S^3\times S^3 $ containing $\mathrm{Ker}(\Phi)$. The isometry $\overline{\Phi}_{p,q}$ normalizes $\Phi(\tilde{G})$  if and only if the following two conditions are satisfied:

\begin{enumerate}
\item $p,q \in  N_{S^3}(R)$;
\item $p,q \in  N_{S^3}(R_K)$;
\item  the following diagram commutes:
\begin{equation} \label{diagramma quadrato reversing}
\begin{gathered}
\xymatrix{
R/R_{K} \ar[d]^-{\beta}\ar[r]^-{\phi} & R/R_{K} \ar[d]^-{\alpha}\\
R/R_{K}   \ar[r]^-{\phi^{-1}} & R/R_{K} \\
}
\end{gathered}
\end{equation}
where $\alpha(xR_K)=p^{-1}xpR_K$ and   $\beta(xR_K)=q^{-1}xqR_K$.
\end{enumerate}
   
\end{Lemma}

Now we want to analyze which groups in  Table~\ref{subgroup} admit an orientation-reversing isometry in their normalizer. The condition given in Lemma~\ref{lemma orientation-reversing primo} excludes all the groups in the families 2, 3, 4, 5, 6, 7, 8, 9, 14, 15, 16, 17, 18, 19 ,23, 24, 29, 34 and  the groups with $n \neq m$ in the families 1, $1^{\prime}$,10, $11$, $11^{\prime}$, 12, 33, $33^{\prime}$.
For  the remaining groups, by Lemma~\ref{lemma orientation-reversing secondo} we obtain that  if $\phi=Id$ then $\overline{\Phi}_{1,1}$ (i.e. the isometry given by the conjugation in $\mathbb{H}$) normalizes  the group and the quotient orbifold admits an orientation-reversing isometry. In these cases the normalizer of $G$ in $\O(4)$  is generated by the normalizer of $G$ in $\SO(4)$ and $\overline{\Phi}_{1,1}$. The element $\overline{\Phi}_{1,1}$ has order two and $\overline{\Phi}_{1,1}\overline{\Phi}_{l,r}\overline{\Phi}_{1,1}^{-1}=\overline{\Phi}_{r,l}$, hence the full isometry group of $S^3/G$ can be easily computed. 

The behavior of families $26^{\prime\prime}$, 32, $32^{\prime}$    33, $33^{\prime}$ is similar, as $\phi^2=Id$ and hence  $\overline{\Phi}_{1,1}$ normalizes again the group. 

The situation for the four remaining families of groups 1, $1^{\prime}$, $11$, $11^{\prime}$ is more complicated. We have already remarked that, if an orientation-reversing element is in the normalizer, then $n=m$. However, in these cases  this necessary condition is not sufficient. 

We explain in detail the situation for Family 1 with $r\geq 2$. In the other remaining cases  the full isometry group can be computed in a very similar way.

\textbf{Family 1.} If $\overline{\Phi}_{p,q}$ normalizes  $G=\Phi((C_{2mr}/C_{2m},C_{2mr}/C_{2m})_s)$, then both $p$ and $q$ normalize $C_{2mr}$. The action by conjugation of $p$ and $q$ on $C_{2mr}$ is either trivial or dihedral. When $p$ and $q$ act in the same way (both trivially or both dihedrally), by Lemma~\ref{lemma orientation-reversing secondo} we obtain that $\phi^2=Id$ and $s^2\equiv_{r} 1$ (i.e. $s^2$ is congruent to 1 $\mod\, r$). In this case $\overline{\Phi}_{1,1}$ normalizes  the group. If $p$ and $q$ act differently, then $\phi^2$ sends an element of $C_{2mr}/C_{2m}$ to its inverse and $s^2\equiv_{r} -1$ (i.e. $s^2$ is congruent to $-1$ $\mod\, r$).  In this case $\overline{\Phi}_{j,1}$ normalizes the group; we remark that  $\overline{\Phi}_{j,1}$ has order 4.

We have $\Isompr(S^3/G)\cong \mathrm{Dih}(S^1\times S^1)$. Since $\Isompr(S^3/G)\cong \Norm_{S^3\times S^3}(\tilde G)/\tilde G$, we represent each isometry as the corresponding coset of $\tilde G$ in $\Norm_{S^3\times S^3}(\tilde G)=(\O(2)^*/S^1,\O(2)^*/S^1).$ The group $\Isompr(S^3/G)$ is generated by the involution $(j,j)\tilde G$ and by the abelian subgroup of index two $N=(S^1\times S^1)/\tilde G$. Both $\overline{\Phi}_{1,1}$ and $\overline{\Phi}_{j,1}$ commute with the element  $(j,j)G$. The group $N$ is isomorphic to $S^1\times S^1$, but the direct factors of the quotient do not correspond in general to the projections of the direct factors of  the original group $S^1\times S^1$. This fact makes the comprehension of the extension of  $N$ by $\overline{\Phi}_{1,1}$ and $\overline{\Phi}_{j,1}$   more complicated. To represent $N$ as the direct product of two copies of $S^1$ we define an isomorphism 
$\gamma: S^1\times S^1 \to N$
by means of the following construction. Let $$\tilde\gamma:\R\times\R\to N=(S^1\times S^1)/\tilde G \qquad\qquad \tilde\gamma({\alpha},{\beta})=\left(e^{i\left(\frac{\alpha}{2mr}+\frac{\beta}{2m}\right)},e^{i\left(\frac{s\alpha}{2mr}+\frac{(s+1)\beta}{2m}\right)}\right)\tilde G\,.$$
It is easy to check that $\mathrm{Ker}(\tilde\gamma)=2\pi\Z\times2\pi\Z$ and thus $\tilde\gamma$ descends to the isomorphism $\gamma:S^1\times S^1\to N$ which can be defined as 
\begin{equation} \label{isomorfismo gamma} \gamma(e^{i\alpha},e^{i\beta})=\left(e^{i\left(\frac{\alpha}{2mr}+\frac{\beta}{2m}\right)},e^{i\left(\frac{s\alpha}{2mr}+\frac{(s+1)\beta}{2m}\right)}\right)\tilde G\,.
\end{equation}

 Since here we have two subgroups isomorphic to $S^1\times S^1$, we need to distinguish the notation in the two cases: an element of $S^1\times S^1 < S^3 \times S^3$ is denoted by $(e^{i a},e^{i b})$ while an element of $N$ is denoted by $(e^{i a},e^{i b})\tilde G$ if it is seen in the quotient  $(S^1\times S^1)/\tilde G$ or by $((e^{i\alpha},e^{i\beta}))=\gamma(e^{i\alpha},e^{i\beta})$ if $N$ is seen as the direct product of two copies of $S^1.$

\textbf{Suppose  that $s^2\equiv_{r} -1 $.}  The isometry $\overline{\Phi}_{j,1}$ normalizes the group; moreover $\overline{\Phi}_{j,1}$ is of order 4 and $\overline{\Phi}_{j,1}^2=\Phi_{j,j}$. By \cite[Proposition 7.12]{farrel-short} the automorphism group of  $S^1\times S^1$ contains only one conjugacy class of order four, so we have a unique semidirect product $\mathbb{Z}_4\ltimes N$ corresponding to the automorphism which maps $((e^{i\alpha}, e^{i\beta}))$ to $((e^{-i\beta}, e^{i\alpha})).$

\textbf{Suppose that $s^2\equiv_{r} 1 $.}  Here the full isometry group  is a semidirect product of $\Isompr(S^3/G)$ with $\mathbb{Z}_2$. In the automorphism group of $(S^1\times S^1)$ we have three classes of involutions which correspond to non-isomorphic semidirect products. In order to determine to which semidirect product $\Isom(S^3/G)$ is isomorphic, we compute the action by conjugation of $\overline{\Phi}_{1,1}$ on $N.$   The element $(e^{i\alpha},e^{i\beta})\tilde G$ is conjugate by $\overline{\Phi}_{1,1}$ to $(e^{i\beta},e^{i\alpha})\tilde G$. To understand which semidirect product we obtain we will use a procedure introduced in \cite{farrel-short}, and to apply this procedure we need to understand the action of $\bar{\Phi}(1,1)$ on $N$ represented as a direct product of two copies of $S^1.$ By using $\gamma$ the action by conjugation of $\overline{\Phi}_{1,1}$ on $N\cong S^1 \times S^1$ is the following:

$$((e^{i\alpha},e^{i\beta}))\longrightarrow ((e^{i(s^2+s-1)\alpha+i(2s+s^2)r\beta},e^{i\frac{1-s^2}{r}\alpha+i(1-s-s^2)\beta}))\,. $$

By applying  the procedure presented  in the proof of  \cite[Proposition 7.9]{farrel-short}, we obtain that if $r$ is odd the automorphism induced by $\overline{\Phi}_{1,1}$ is conjugate to the following automorphism:

$$((e^{i\alpha},e^{i\beta}))\longrightarrow ((e^{i\beta},e^{i\alpha}))\,,$$
while if $r$ is even $\overline{\Phi}_{1,1}$ is conjugate to the following automorphism:

$$((e^{i\alpha},e^{i\beta}))\longrightarrow ((e^{-i\alpha},e^{i\beta}))\,.$$

\begin{table}
\begin{adjustbox}{width=\textwidth}
\centering
\begin{tabular}{|l|p{8cm}|c||l|c|}%\label{tabella-gruppi}
\hline
 $G$  & case & $\Isom(S^3/G)$ & $G$  & $\Isom(S^3/G)$  \\
\hline
 1 &   $m=n,$ $r>2,$ $s^2\equiv_{r} 1 $ and   $r$ odd   &$(\mathbb{Z}_2\times \mathbb{Z}_2) \ltimes_{\alpha_{1}} (S^1\times S^1)$ & 12 & $\Z_2 \times \Z_2$    \\   
1 &  $m=n,$  $r>2,$ $s^2\equiv_{r} 1 $ and $r$ even & $(\mathbb{Z}_2\times \mathbb{Z}_2) \ltimes_{\alpha_{2}} (S^1\times S^1)$ &  20 & $D_8$  \\  
1 & $m=n,$ $r>2$ and $s^2\equiv_{r} -1 $  &$\mathbb{Z}_4 \ltimes_{\alpha_{3}} (S^1\times S^1)$ & 21 & $\Z_2 \times \Z_2$   \\
1 &  $m=n,$ $r=1$ and  $m>1$  &  $\Z_2\ltimes_{\alpha_{4}} (\O(2)\times \O(2))$    &   $21^{\prime}$ &  $(\Z_2)^3$   \\
1 &  $m=n,$ $r=1$ and  $m>1$  &  $\Z_2\ltimes_{\alpha_{4}} (\O(2)\widetilde{\times}\O(2))$ & 22  & $\Z_2 \times D_6$   \\      
1 & $m=n=2$ and $r=1$  &  $\PO(4)$  &  25 & $\Z_2$   \\       
$1^{\prime}$ &  $m=n,$ $r>2,$ $s\equiv_{r} 1$ and   $r$ odd   &$(\mathbb{Z}_2\times \mathbb{Z}_2) \ltimes_{\alpha_{1}} (S^1\times S^1)$   & 26 & $\Z_2$    \\
$1^{\prime}$ & $m=n,$ $r>2,$ $s\equiv_{r} 1$  and $r$ even & $(\mathbb{Z}_2\times \mathbb{Z}_2) \ltimes_{\alpha_{2}}  (S^1\times S^1)$  &  $26^{\prime}$  & $\Z_2 \times \Z_2$    \\
$1^{\prime}$ & $m=n,$ $r>2$ and  $s\equiv_{r} -1$  &$\mathbb{Z}_4 \ltimes_{\alpha_{3}}  (S^1\times S^1)$   &  $26^{\prime\prime}$  & $\Z_2 \times \Z_2$    \\
$1^{\prime}$ & $m=n,$ $r=2$ and  $m>1$  &  $\Z_2\ltimes_{\alpha_{4}} (\O(2)^*\widetilde{\times}\O(2)^*)$ &  27  & $D_6$     \\  
10 & $m=n$ and   $m>2$ & $D_{8}$  & 28 & $\Z_2 \times \Z_2$     \\
10 &  $m=n=2$ & $\Z_2\ltimes_{\alpha_{4}} (D_{6} \times D_{6})$ &  30 & $\Z_2$    \\
11 & $m=n,$ $r>2$ and $s^2\equiv_{r} \pm 1$ &$\Z_2 \times \Z_2$  & 31   & $\Z_2$   \\
11 & $m=n,$ $m>1$ and $r=1$  &$\Z_2 \times D_8$ & $31^{\prime}$ & $\Z_2 \times \Z_2$   \\ 
11 & $m=n,$ $m>1$ and $r=2$  &$(\Z_2)^4$ & $32$  & $\Z_2$   \\  
11 & $m=n=2$ and $r=2$ &$\Z_2 \times O$  & $32^{\prime}$ &  $\Z_2 \times \Z_2$   \\
$11^{\prime}$  & $m=n,$ $r>2,$ $s^2\equiv_{r} 1$,  and $(s^2-1)/r$ even &$\Z_2 \times \Z_2 \times \Z_2$ &  $33$   & $\Z_2 $     \\ 
$11^{\prime}$  & $m=n,$ $r>2,$ $s^2\equiv_{r} 1$ and $(s^2-1)/r$ odd & $D_8$ &   $33^{\prime}$ & $\Z_2 \times \Z_2$     \\ 
%$11^{\prime}$  & $m=n,$ $r>2,$ $s^2\equiv_{r} -1$ and $(s^2+1)/r$ even &$\Z_2 \times %
%\Z_2 \times \Z_2$   &  &    \\
$11^{\prime}$  & $m=n,$ $r>2$ and  $s^2\equiv_{r} -1$  & $D_8$ &&  \\
$11^{\prime}$  & $m=n=1$ and $r=4$  &$\Z_2 \times \Z_2 \times D_6$   &&  \\
\hline
\multicolumn{5}{|p{17,3cm}|}{$\alpha_1:$ the group $\Z_2 \times \Z_2$ is generated by $f$ and $g$ s.t. $\alpha_1(f)((e^{i\alpha},e^{i\beta}))=((e^{-i\alpha},e^{-i\beta}))$ and $\alpha_1(g)((e^{i\alpha},e^{i\beta}))=((e^{i\beta},e^{i\alpha})).$} \\
\multicolumn{5}{|p{17,3cm}|}{$\alpha_2:$ the group $\Z_2 \times \Z_2$ is generated by $f$ and $g$ s.t. $\alpha_2(f)((e^{i\alpha},e^{i\beta}))=((e^{-i\alpha},e^{-i\beta}))$ and $\alpha_2(g)((e^{i\alpha},e^{i\beta}))=((e^{-i\alpha},e^{i\beta})).$} \\
\multicolumn{5}{|p{17,3cm}|}{$\alpha_3:$ the group $\Z_4$ is generated by $f$ s.t. $\alpha_3(f)((e^{i\alpha},e^{i\beta}))=((e^{-i\beta},e^{i\alpha})).$}\\
\multicolumn{5}{|p{17,3cm}|}{$\alpha_4:$ $\Z_2$ acts on a product, direct or central, whose  generic element can be represented by a couple $(x,y)$; the non trivial element in $\Z_2$ maps $(x,y)$ to $(y,x)$.}\\
\hline
\end{tabular}
\end{adjustbox}
\bigskip
\caption{Table of full isometry groups when $\Isom(S^3/G)\neq\Isompr(S^3/G)$ }
\label{tableisometryreversing}
\end{table}

%, is isomorphic to $\Norm_{\SO(4)}(G)/G$

%$\Norm_{S^3\times S^3}(\tilde G)/\tilde G$
% $$\Norm_{S^3\times S^3}(\tilde G)=(\O(2)^*/S^1,\O(2)^*/S^1)$$
%unless $r=2$ (or $r=1$). The isometry group $\Norm_{S^3\times S^3}(\tilde G)/\tilde G$ is a dihedral extension of $(S^1\times S^1)/\tilde G$, the latter being again isomorphic to $S^1\times S^1$. Hence we get
%$\Isompr(S^3/\tilde G)\cong \mathrm{Dih}(S^1\times S^1)$. The same result is recovered analogously for Family 1', unless $r=2$.

% $\Norm_{\SO(4)}(G)/G$

%------------------------------------SECTION FIBRATIONS-----------------------------

\section{Generalized Smale conjecture} \label{sec gen smale}

The purpose of this section is to provide a proof of the $\pi_0$-part of the Generalized Smale Conjecture for spherical compact 3-orbifolds.

\begin{Theorem}[$\pi_0$-part of the Generalized Smale Conjecture for spherical 3-orbifolds]\label{smale conjecture}
Let $\OO$ be any compact three-dimensional spherical orbifold. The inclusion $\Isom(\OO)\rightarrow\Diff(\OO)$ induces a group isomorphism
$$\pi_0 \Isom(\OO)\cong \pi_0 \Diff(\OO)\,.$$
\end{Theorem}

Before giving the proof, we will need to recall some notions on Seifert fibrations for orbifold, and to prove some preliminary results.

\subsection{Definition of Seifert fibrations for orbifolds}
A Seifert fibration of a 3-orbifold $\mathcal O$ is a projection map $\pi:\mathcal O \rightarrow \mathcal B$, where $\mathcal B$ is a 2-dimensional orbifold, such that for every point $x\in\mathcal  B$ there is an orbifold chart $x\in U\cong \tilde U/\Gamma$, an action of $\Gamma$ on $S^1$ (inducing a diagonal action of $\Gamma$ on  $\tilde U\times S^1$)  and a diffeomorphism $\psi:(\tilde U\times S^1)/\Gamma\rightarrow \pi^{-1}(U)$  which makes the following diagram commute:

\[
\xymatrix{
\pi^{-1}(U) \ar[dr]_-{\pi} & & (\tilde{U}\times S^1)/ \Gamma \ar[ll]_-{\psi} \ar[dl] & \tilde{U}\times S^1 \ar[l] \ar[dl]^-{\mbox{pr}_1} \\
 & U\cong\tilde{U}/ \Gamma  &  \tilde{U} \ar[l] &
}
\]

If we restrict our attention to orientable 3-orbifolds $\mathcal O$, then the action of $\Gamma$ on $\tilde U\times S^1$ needs to be orientation-preserving. In this case, we will consider a fixed orientation both on $\tilde U$ and on $S^1$. Every element of $\Gamma$ may preserve both orientations, or reverse both.

The fibers $\pi^{-1}(x)$ are simple closed curves or intervals. If a fiber projects to a non-singular point of $\mathcal B$, it is called generic. Otherwise we will call it exceptional.

Let us define the local models for an oriented Seifert fibered orbifold. Locally the fibration is given by the curves  induced on the quotient $(\tilde U \times S^1)/\Gamma$ by   the standard fibration of $\tilde U\times S^1$  given by the curves $\{y\}\times S^1$.

If the fiber is generic, it has a tubular neighborhood with a trivial fibration. 
When $x\in \mathcal B$ is a cone point labelled by $q$, the local group $\Gamma$ is a cyclic group of order $q$ acting orientation preservingly on $\tilde U$ and thus it can act on $S^1$ by rotations. Hence a fibered neighborhood of the fiber $\pi^{-1}(x)$ is a fibered solid torus. One can define the \emph{local invariant} of the fiber $\pi^{-1}(x)$ as the ratio $p/q\in\mathbb{Q}/\mathbb{Z}$, where a generator of $\Gamma$ acts on $\tilde U$ by rotation of an angle ${2\pi}/{q}$ and on $S^1$ by rotation of $-{2\pi p}/{q}$ -- however the study of local invariants is not one of the main purposes of this paper.
%Then we define the local invariant associated to $x$ to be the ratio $p/q\in\mathbb{Q}/\mathbb{Z}$.  We remark that in the literature different sign conventions are used, we use the same as in \cite{BS} while in 
%\cite{Dun1} the invariant is defined to be $-p/q.$ In the orbifold context $p$ and $q$ are not necessarily coprime. In this section  the local invariants $p/q$ have to be considered normalized so that $0\leq p<q$. In the formulae we compute in Sections~\ref{abelian} and \ref{remaining} we give the local invariants in a non-normalized form.   
The fiber $\pi^{-1}(x)$ may be singular (in the sense of orbifold singularities) and the  index of singularity is $\gcd(p,q)$. If $\gcd(p,q)=1$ the fiber is not singular.  Forgetting the singularity of the fiber (if any), the local model  coincides with the  local model of a Seifert fibration for manifolds.  
%\begin{figure}[htb]
%\begin{center}
%\includegraphics[height=4cm]{fibered-torus3.eps}
%\caption{A fibered neighborhood of an exceptional fiber of invariant 1/3 corresponding to a cone point}\label{cone-point}
%\end{center}
%\end{figure}

If $x$ is a corner reflector, namely $\Gamma$ is a dihedral group, then the non-central involutions in $\Gamma$ need to act on $\tilde U$ and on $S^1$ by simultaneous reflections. Here the local model is the so-called solid pillow, which is a topological 3-ball with some singular set inside. There is an index two cyclic subgroup of $\Gamma$, acting as we have previously described. Again, the local invariant associated to $x$ can be defined as the local invariant ${p}/{q}$ of the cyclic index two subgroup, and the fiber $\pi^{-1}(x)$ has singularity index $\gcd(p,q)$. The fibers of $U\times S^1$ intersecting the axes of reflections of $\Gamma$ in $\tilde U$ project to segments that are exceptional fibers of the 3-orbifold; the other fibers of $\tilde U\times S^1$  project to simple closed curves.  

%In Figure~\ref{corner-reflector} the  horizontal segments are not fibers but consist of the endpoints of the fibers that are segments; they are singular (in the sense of orbifold singularities) of index two.

%\begin{figure}[htb]
%\begin{center}
%\includegraphics[height=5cm]{solid-pillows-3.eps}
%\caption{Two copies of a fibered neighborhood of an exceptional fiber of invariant 1/2 corresponding to a corner reflector}\label{corner-reflector}
%\end{center}
%\end{figure}

Finally, over mirror reflectors (local group $\mathbb{Z}_2$), we have a special case of the dihedral case. The local model is topologically a 3-ball with two disjoint singular arcs of index 2. More details can be found in  \cite{bonahon-siebenmann} or \cite{dunbar2}.

There is a classification theorem for Seifert fibered 3-orbifolds up to orientation-preserving, fibration-preserving diffeomorphisms by means of some invariants. We won't use this theorem here, so we don't provide details. We only briefly remind that 
an oriented  Seifert fibered orbifold is  determined up to  diffeomorphisms which preserve the orientation and the fibration by the data of the base orbifold, the local invariants associated to cone points and corner reflectors, an additional invariant $\xi\in\mathbb{Z}_2$ associated to each boundary component of the base orbifold and the Euler number. If we change the orientation of the orbifold, then the sign of local invariants and  Euler number are inverted. %The normalized local invariants pass in this case from $p/q$ to $(q-p)/q$. 
For the formal definitions of Euler number and of invariants associated to boundary components, as well as the complete statement and proof of the classification theorem, we refer again  to \cite{bonahon-siebenmann} or \cite{dunbar2}.

\subsection{Seifert fibrations of $S^3$} \label{sec seifert s3}

Seifert fibrations of $S^3$ are well known: it is proved in \cite{seifert} that, up to an orientation-preserving diffeomorphism, they are  given by  the maps of the form  $\pi:S^3\rightarrow S^2\cong \mathbb{C}\cup\left\{\infty\right\}$ $$\pi(z_1+z_2 j)=\frac{z_1^u}{z_2^v}\qquad\textrm{or}\qquad\pi(z_1+z_2 j)=\frac{\overline{z}_1^u}{z_2^v}$$ for $u$ and $v$ coprimes. We call \textit{standard} the Seifert fibrations given by these maps; the  classification of Seifert fibrations of $S^3$ can thus be rephrased in the following way: each Seifert fibration of $S^3$ can be mapped by an orientation-preserving diffeomorphism of  $S^3$ to a standard one.

The base orbifold of a Seifert fibration of $S^3$ is $S^2$ with two possible cone points. When $u=v=1$, $\pi(z_1+z_2 j)={z_1}/{z_2}$ is called the Hopf fibration. In this case the base orbifold is $S^2$ and all the fibers are generic. %;  if we consider  the  orientation of $S^3$ induced by the standard orientation of $\C \times \C$, the  Euler number of the Hopf fibration is $-1$.
The projection $\pi:S^3\to S^2$ of the Hopf fibration is also obtained as the quotient by the following $S^1$-action on $S^3$: an element $w=x+iy\in S^1$ acts on $S^3$ simply as left multiplication by $w$. That is, $w\cdot(z_1+z_2j)=(wz_1+wz_2j)$. 

The only other Seifert fibration whose  fibers are all generic is given by $\pi(z_1+z_2 j)={\overline{z}_1}/{z_2}$; we call {\it anti-Hopf} this fibration.

%It is known (see \cite[Theorem 5.1]{DM}) that a  Seifert spherical 3-orbifold $S^3/G$ is  isometric to an orbifold $S^3/G'$ where $G'$ is a subgroup of  $\SO(4)$ respecting the Hopf fibration; the isometry may be orientation reversing. QUALCHE DETTAGLIO?

{{It turns out that,}} in Du Val's list, the subgroups of $\SO(4)$ which preserve the Hopf fibration are those with $L=C_{m}$ or $L=D_{2m}^*$, for some $m$. 
The other fibrations of type $\pi(z_1+z_2 j)={\overline{z}_1^u}/{z_2^v}$ are left invariant only by groups in Family 1,1',11,11' and the spherical orbifolds obtained as quotients by these groups have an infinite number of nonisomorphic fibrations. {{See \cite{mccullough} or \cite{mecchia-seppi}.}}
For the remaining fibrations, it suffices to note that the orientation-reversing isometry $\overline{\Phi}_{1,1}$ maps the fibration $\pi(z_1+z_2 j)={\overline{z}_1^u}/{z_2^v}$ to $\pi(z_1+z_2 j)={z_1^u}/{z_2^v}$.  The isometry $\Phi_{l,r}$ preserves a fibration $\pi(z_1+z_2 j)={\overline{z}_1^u}/{z_2^v}$ if and only if $\overline{\Phi}_{1,1}^{-1} \Phi_{l,r}\overline{\Phi}_{1,1}=\Phi_{r,l}$ preserves $\pi(z_1+z_2 j)={z_1^u}/{z_2^v}$.

We remark that finite subgroups in Du Val's list can leave  invariant  also  fibrations that are not  standard, thus obtaining different fibrations on the same spherical orbifold.  The  following Lemma shows that this phenomenon can occur only in some specific cases.

\begin{Lemma} \label{lemma due casi}
Let $G$ be a finite subgroup of $\SO(4)$ leaving invariant a Seifert fibration $\pi$ of $S^3$, then one of the two following conditions is satisfied:
\begin{enumerate}
\item $G$ is conjugate in $\SO(4)$ to a subgroup in Families $1$, $1'$, $11$ or $11'$;
\item there exists  an orientation-preserving  diffeomorphism $f:S^3\to S^3$ such that   $\pi\circ f$ is the Hopf or the anti-Hopf fibration and $f^{-1}Gf$ is a subgroup of $\SO(4).$
\end{enumerate}
\end{Lemma}

\begin{proof}
The fibration $\pi$ is mapped by an orientation-preserving diffeomorphism to a standard Seifert fibration of $S^3.$ 
We recall that the fibrations  $\pi(z_1+z_2 j)={z_1^u}/{z_2^v}$ and $\pi(z_1+z_2 j)={\overline{z}_1^u}/{z_2^v}$ with $(u,v)\neq (1,1)$ have exactly two exceptional fibers which have different invariants.  
If $\pi$ is mapped to one of these Seifert fibrations, the group $G$ must leave invariant both exceptional fibers. A finite 
group of isometries leaving invariant a simple closed curve is isomorphic to a subgroup of Dih($\mathbb{Z}_n\times \mathbb{Z}_m$), a dihedral 2-extension of an abelian group of rank at most two (see for example \cite[Lemma 1]{mecchia-zimmermann}). The only groups in Du Val's list with this property are in Families 1, $1'$, $11$ or $11'$ and thus case 1 occurs.

Now we can suppose that $\pi$ is mapped  by an orientation-preserving diffeomorphism $g$ either to the Hopf fibration or to the anti-Hopf fibration. 
Let us first consider the case of the Hopf fibration. In this case $G$ is conjugate by  $g$ to $G'$, a finite group of diffeomorphisms of $S^3$ leaving invariant the Hopf fibration. By the proof of \cite[Theorem 5.1]{morgan-davis} we can conjugate  $G'$ to a group of isometries by using a diffeomorphism $h$ which leaves invariant the Hopf fibration. The diffeomorphism $h\circ g$  is the $f$ we are looking for. 
We can reduce the anti-Hopf case to the previous one by conjugating by $\overline{\Phi}_{1,1}.$
\end{proof}

\subsection{Proof of main result} In this subsection, we will provide the proof of Theorem 1. Before that, we need an additional preliminary Lemma. \\

Hence, let $\OO=S^3/G$  be a spherical 3-orbifold, where   $G$ is a finite subgroup of $\SO(4)$. We denote by $\Sigma$ the singular set of $\OO$ and by $M$ the complement of $\overset{\circ}{N}(\Sigma)$ in $\OO$, where $\overset{\circ}{N}(\Sigma)$ is  the interior of a regular neighbourhood of $\Sigma$.

\begin{Lemma}\label{Seifert-complement}
If $M$ is a Seifert fibered manifold, then $G$ is conjugate to one of the groups in Families from 1 to 9 (including $1'$) or in Family 34. 
\end{Lemma}

\begin{proof}

Since $M$ has a Seifert fibration for manifolds, then $\Sigma$ is a link.

Lift the fibration of $M$ to $S^3\setminus \overset{\circ}{N}(\tilde{\Sigma})$ where $\tilde{\Sigma}$ is the preimage in $S^3$ of $\Sigma$ under the projection $S^3\rightarrow S^3/G$ and $\overset{\circ}{N}(\tilde{\Sigma})$ is the interior of a regular neighbourhood of $\tilde{\Sigma}$;  we note that $\tilde{\Sigma}$  is the set of points in $S^3$ that are fixed by a non-trivial element of $G$. 
The fibration of $S^3\setminus \overset{\circ}{N}(\tilde{\Sigma})$ extends to a fibration of $S^3$ invariant under the action of $G$ unless  $\tilde{\Sigma}$ is the link in Figure~\ref{singular-link} (see the proof of \cite[Theorem 1]{burde-murasugi}).

\begin{figure}[htb]
\begin{center}
\includegraphics[height=4cm]{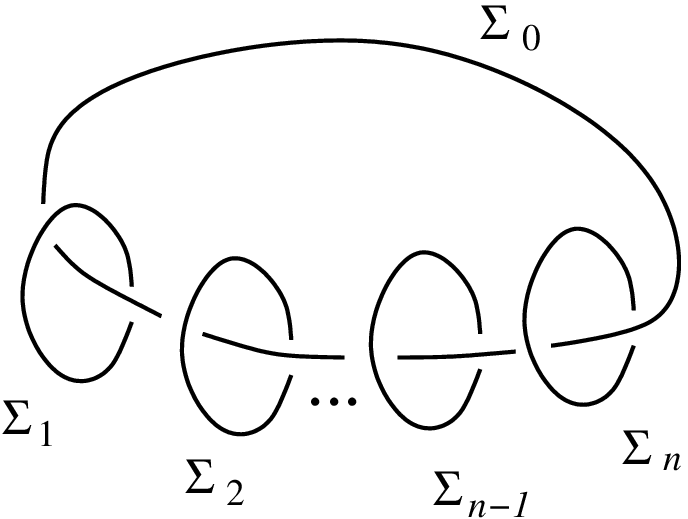}
\caption{}\label{singular-link}
\end{center}
\end{figure}

\textit{Case 1: the fibration of $S^3\setminus \overset{\circ}{N}(\tilde{\Sigma})$ extends.} In this case the Seifert fibration of $M$ is induced by a Seifert fibration of $S^3$ invariant under the action on $G$ and extends to  $\OO$.  We remark that the fibration on $M$ is a classical Seifert fibration for manifold, while the fibration on $\OO$ is  a Seifert fibration in an orbifold sense. The base orbifold of $\OO$ contains mirror reflectors or corner reflectors if and only if the fibration of $\OO$ contains infinite fibers that are not closed curves (they are arcs). Hence the fibration of $\OO$ induces a  Seifert fibration on the manifold $M$ when  the base orbifold has only cone points as singularities.  

%An analysis of  which Seifert fibrations of $S^3$  are left invariant by a finite subgroup of $S^3$ is presented in \cite[Section 3]{mecchia-seppi} (see also \cite{dunbar}.) The groups in the Families from 20 to 32 (including $21',$ $26',$ $26'',$ $31'$ and $32'$) do not preserve any fibration of $S^3$, so they cannot occur. We remark that, to analyze which fibration are left invariant by a group, we have to distinguish between $(L,L_K,R,R_K,\phi)$ and $(R,R_K,L,L_K,\phi^{-1})$. These two groups  are not in general conjugated in $\SO(4)$ so their behaviour with fibrations are different, but they are conjugated in $\mathrm{O}(4)$ and for this reason in Table~\ref{subgroup} only one group between  $(L,L_K,R,R_K,\phi)$ and $(R,R_K,L,L_K,\phi^{-1})$ is listed.  

%A group $(L,L_K,R,R_K,\phi)$  leaves invariant the Hopf fibration if and only if $L$ is $C_n$ or $D^*_{2n}$. The Seifert invariants of the fibration of  $S^3/G$ induced  by the Hopf fibration of $S^3$ are described in \cite[Table 2,3,4]{mecchia-seppi}. 
%The mirror image of the Hopf fibration (which we call the  anti-Hopf fibration) is left invariant by  $(L,L_K,R,R_K,\phi)$ if and only if $R$ is $C_n$ or $D^*_{2n}$. The Seifert invariant of the quotient corresponding to the anti-Hopf fibration can be easily computed by \cite[Table 2,3,4]{mecchia-seppi}, since the Seifert invariants of $S^3/(L,L_K,R,R_K,\phi)$ induced by the  the anti-Hopf fibration equal the Seifert invariants of $(R,R_K,L,L_K,\phi^{-1})$ induced by the Hopf fibration.
Lemma \ref{lemma due casi} ensures that the spherical orbifolds only have {{fibrations induced by an isometric copy of the Hopf fibration (including the anti-Hopf fibration),}} except for groups which belong to Families $1$, $1'$, $11$, $11'$ up to conjugation.  
We recall that the fibrations of $S^3$ of the form $z_1^u/z_2^v$ or $\bar z_1^u/z_2^v$, with $(u,v)\neq (1,1),$ have two exceptional fibers, each left invariant by the group. For groups in Families $11$, $11'$, there must be an involution which acts as a reflection on the exceptional fibers, and therefore the base orbifold of the quotient contains corner reflectors. 
%We recall that we look for Seifert fibrations of the quotient orbifold without mirror and corner reflector. 
%If an orbifold has a Seifert fibration induced from the fibrations of $S^3$ of the form $z_1^u/z_2^v$ or $\bar z_1^u/z_2^v$, with $$
Finally, by analyzing {{Table 2, Table 3 and Table 4 provided in \cite{mecchia-seppi}}}, one sees that if a spherical orbifold $S^3/G$ has (at least) one Seifert fibration with only cone singular points in the base orbifold, then it is conjugate to a group in Families 1 to 9 (including $1'$) or in Family 34. 
%Analyzing \cite[Table 2,3,4]{mecchia-seppi} we can conclude that $M$ has a Seifert fibration induced by the Hopf or by the anti-Hopf fibration if and only if  $G$ is in the family from 1 to 9 (including $1'$) or in Family 34.
	%In principle a Seifert fibration of $M$ can be induced also  by one of the other fibrations of $S^3$ (the fibrations whose base orbifold is not the 2-sphere but an orbifold). These fibrations are left invariant only by the groups in Families 1, $1'$,$11$ and $11'.$   It is easy to compute that for the last two families the base orbifold of $\OO$ corresponding to any fibration of $S^3$ contains mirror reflector. Therefore the groups in Families 11 or $11'$ can be excluded. 

\medskip

\textit{Case 2: $\tilde{\Sigma}$ is the link in Figure~\ref{singular-link}}. In this case $\Sigma_0$ is the fixed point set of a non-trivial element of $G$. Since $\Sigma$ admits no singular point with dihedral local group, the subgroup of $G$ leaving invariant $\Sigma_0$ is cyclic or the direct product of two cyclic groups (see for instance \cite[Lemma 1]{mecchia-zimmermann}) and can be conjugated to a group in Familiy 1 or $1'$.  If $n\neq 1$, the whole group $G$ leaves invariant $\Sigma_0$ and we are done. In the case of $n=1$, $G$ contains a subgroup $G_0$ of index at most two leaving invariant $\Sigma_0$. If $G=G_0$, then  the group  can be conjugated to a group in Familiy 1 or $1'$. Otherwise the elements in $G\setminus G_0$ exchange $\Sigma_0$ with $\Sigma_1$. The groups having these properties are in Families  2, 3, 4 and 34, {{see again \cite{mecchia-seppi}.}}
\end{proof}

We are now ready to give the proof of our main result.

\begin{proof}[Proof of Theorem 1]
%\[
%\xymatrix{
%1 \ar[r] & \Diff_0(\OO) \ar[r] & \Diff(\OO) \ar[r] & \Out(G)
%}
%\]
%We consider the following commutative diagram

%\[
%\xymatrix{
%1 \ar[r] & \Diff_0(\OO) \ar[r] & \Diff(\OO) \\
%1 \ar[r] & \Isom_0(\OO) \ar[u] \ar[r] & \Isom(\OO) \ar[u] \\
%}
%\]
 
%which provides an induced map 

%\[
%\xymatrix{
%\pi_0 \Isom(\OO)=\Isom(\OO)/\Isom_0(\OO) \ar[r]^-{\iota} & \Diff(\OO)/\Diff_0(\OO)= \pi_0 \Diff(\OO)\,.
%}
%\]

We denote by $\iota: \pi_0 \Isom(\OO) \rightarrow \pi_0 \Diff(\OO) $ the homomorphism induced by the inclusion  $\Isom(\OO)\rightarrow\Diff(\OO).$

For spherical manifolds the theorem was proved in \cite{mccullough}. Therefore we can suppose that $\Sigma$ is not empty and we can apply the results in \cite{cuccagna-zimmermann}  where the authors proved   the existence of  a finite subgroup of diffeomorphisms $H$ such that the standard projection  $\Diff(\OO) \rightarrow \pi_0 \Diff(\OO) $ restricted to $H$ is surjective. As a consequence of the Thurston Orbifold Geometrization Theorem (see \cite{boileau-leeb-porti} and \cite{dinkelbach-leeb}), we can suppose that $H$ is a group of isometries of $\OO$, and we can conclude that also $\iota$ is surjective.

If $\Sigma$ is a link, then \cite[Theorem 2]{gordon-litherland} implies directly that $M=\OO\setminus \overset{\circ}{N}(\Sigma)$  is  irreducible and atoroidal (i.e. each incompressible torus is boundary parallel). Indeed the argument used in the  proof of \cite[Theorem 2]{gordon-litherland} works also if $\Sigma$ is not a link, so we have that in any case $M$ is irreducible and atoroidal. 
By  the geometrization of 3-manifolds with non-empty boundary (see for example \cite[Proposition 3]{shalen}) we obtain that $M$ is either hyperbolic or Seifert fibered.  

If $M$ is hyperbolic, then  by \cite[Theorem 1]{cuccagna-zimmermann}  the homomorphism $\iota$ is also injective and we are done. We  can suppose that $M$ is Seifert fibered and by Lemma~\ref{Seifert-complement} the group $G$ is  in one of the Families from 1 to 9 (including $1'$) or in Family $34$.

%The singular sets of the  non-fibered spherical orbifolds are described in \cite{dunbar}and none of them is a link.  Therefore $G$ leaves invariant a Seifert fibration of $S^3$ and $\OO$ is a Seifert fibered orbifold.

An orientation-reversing isometry cannot be homotopic to the identity, hence to prove that  $\iota$ is injective we can reduce to the orientation-preserving case. We denote by $\Diff^+\!(\OO)$ the group of orientation-preserving diffeomorphism of $\OO$.

From now on we will identify the orbifold fundamental group $\pi_1(\OO)$ with $G$.
We denote by $\Out(G)$ the outer automorphism group, namely the quotient of the group of automorphisms of $G$ by the normal subgroup of inner automorphisms. By using the properties of the universal covering orbifolds (see for example \cite[Sections 4.6 and 4.7]{choi}), we will define a homomorphism $\beta:\Diff^+\!(\OO) \rightarrow \Out(G)$. If $f\in \Diff^+\!(\OO)$, then $f$ can be lifted to a diffeomorphism $\tilde f$ of $S^3$ that normalizes $G$ and by conjugation induces an automorphism of $G$. Different choices of the lifting can give different automorphisms, but they  coincide up to the composition with an inner automorphism. Thus we can define $\beta(f)$ as the outer automorphism induced by a lift of $f$. Since an element of $\Diff^+\!(\OO)$ isotopic to the identity lifts to a diffeomorphism of $S^3$ which is $G$-equivariantly isotopic to the identity, an element of $\Diff^+_0\!(\OO) $ induces a trivial   automorphism on $G.$ Thus $\Diff^+_0\!(\OO) $ is contained in the kernel of $\beta$ and we obtain an induced  homomorphism $\pi_0\Diff(\OO) \rightarrow \Out(G)$ that we denote by $\alpha.$

Therefore we have the following  composition of group homomorphisms
\[
\xymatrix{
\pi_0\Isompr(\OO) \ar[r]^-{\iota} & \pi_0\Diff^+\!(\OO) \ar[r]^-{\alpha} & \Out(G)\,.
}
\]

{{We claim that, when $\OO=S^3/G$ and $G$ is in Families 1-9 and 34 (see Tables~\ref{tableisometry} and ~\ref{tableisometryexceptions}),  the homomorphism $\alpha \circ \iota$ is injective, with the exception of two groups: $(C_4/C_1,C_4/C_1)$ and $(C_4/C_2,C_4/C_2).$ %If an element of $\Isompr(\OO)$  induces a trivial outer automorphism on $G$, we can find an isotopy from this element to the identity such that each level of the isotopy is an isometry. 
For clarity, we give a proof of this claim in a separate lemma (Lemma \ref{lemma trivial action} below). If $\alpha\circ\iota$ is injective, then  $\iota$ is injective, and thus this will conclude the proof for all groups but the two exceptional ones.

We consider now the two remaining cases $(C_4/C_1,C_4/C_1)$ and $(C_4/C_2,C_4/C_2)$. For both groups there exists exactly one non-trivial element  of $\pi_0\Isompr(\OO)$ inducing a trivial outer automorphism of $G$; this is induced by the element $(j,j)\in \Norm_{S^3\times S^3}(\tilde G)$. In both cases $\OO$ has a knot as singular set (see \cite[Table 2 and 3]{mecchia-seppi}). The isometry of $\OO$ induced by $(j,j)$ acts as a reflection  on the singular knot (see \cite[page 831]{mecchia-seppi} where the Families 11 and $11'$ are analyzed).
This implies that  the isometry induced by $(j,j)$ can not be  isotopic to the identity  in $\Diff^+\!(\OO)$; in fact, if it was, restricting the isotopy to the singular set of $\OO$, one would obtain an isotopy between a reflection of $S^1$ and the identity, and this is impossible. We can conclude also for these two groups that the homomorphism $\iota$ is injective.
}}
\end{proof}

\begin{Lemma} \label{lemma trivial action}
Let $G$ be in Families 1-9 and 34  and let $\OO=S^3/G$. If $G$ is not isomorphic to  $(C_4/C_1,C_4/C_1)$ and $(C_4/C_2,C_4/C_2)$, then $\alpha\circ\iota$ is injective.
\end{Lemma}

\begin{proof}

%We will prove that any element of $\Norm_{S^3\times S^3}(\tilde G)\setminus \tilde G$ which is not isotopic to the identity in $\Norm_{S^3\times S^3}(\tilde G)$ induces a non-trivial action on $\tilde G$ by conjugation.

Let us first consider the  groups  in Table~\ref{tableisometry} (where the small indices are excluded).

In the families $1$ and $1'$ the non trivial element in $\pi_0\Isompr(\OO)$  is induced by $(j,j)\in \Norm_{S^3\times S^3}(\tilde G).$ The automorphism induced by $(j,j)$ on $G$ sends each element to its inverse. Since $G$ is an abelian group, the automorphism  induced by $(j,j)$  is non-trivial, hence non-inner, unless all the elements of the group have order $2$. Moreover, since the groups $G$ in the families $1$ and $1'$ are generated by two elements, this can happen only if $G$ has order $2$ or $4$. In Table~\ref{tableisometry} we consider $C_{2m}\neq C_2$ and $r>2$, therefore the only groups of order $4$ or $2$ are $(C_4/C_1,C_4/C_1)$ and $(C_8/C_1,C_8/C_1)$. The latter is cyclic of order four, so again the  induced automorphism in not trivial. The only exceptional case here is $(C_4/C_1,C_4/C_1).$

We consider now  Families 2-9 and  Group 34. Each element of  $\pi_0\Isompr(\OO)$  is induced by one element in $\Norm_{S^3\times S^3}(\tilde G)$ of one of these three types: $(j,1)$, $(1,f)$ or $(j,f)$ where $f$ is an element of $\Norm_{S^3}(R)$ inducing a non-inner automorphism on $R/C_2$, the quotient of $R$ by the center of $S^3.$

Let $\alpha$  an element of $\tilde G$ of type  $(1,f)$ or $(j,f)$; the element  $\alpha$  induces by conjugation an automorphism on the quotient $G/\Phi(L_K\times \{1\})$. We can construct an isomorphism between  $G/\Phi(L_K\times \{1\})$ and $R/C_2$ sending $\Phi(l,r)$ to $r;$ the automorphism induced by $\alpha$ corresponds through this isomorphism to the automorphism induced by $f$ on  $R/C_2$, which is non-inner. This implies that the automorphism induced by $\alpha$ in $G$ is not inner.

Let us now consider the elements of type $(j,1)$. Since we are restricting (for the moment) to Families 2-9 and  Group 34 in Table~\ref{tableisometry} (thus excluding small indices), in each group there exists an element of type $\Phi(l,1)$ where $l$ has order at least four. 
The automorphism induced by  $(j,1)$ maps $\Phi(l,1)$ to $\Phi(l^{-1},1)$. Since $(l,1)\neq\pm (l^{-1}, 1)$ the action of $(j,1)$ on $\Phi(l,1)$ is not trivial. Therefore the automorphism on $G$ can not be inner, because $\Phi(l,1)$ is in the center of $G$.

The analysis of the groups with small indices  that are listed  in Table~\ref{tableisometryexceptions} is similar. For Families 2, 5, 7 and 9 of Table~\ref{tableisometryexceptions}, the proof is exactly the same as above, with the additional remark that if $L=C_2$, non-trivial elements of  $\pi_0\Isompr(\OO)$ cannot be induced by $(j,1)$ since $(j,1)$ is isotopic to the identity in $\Norm_{S^3\times S^3}(\tilde G)$. 

For the small indices of Families $1$ and $1'$, the non-trivial elements of $\pi_0\Isompr(\OO)$ are induced by one of the following three cases: $(j,j)$, $(j,1)$ and $(1,j)$. The case $(j,1)$ is treated exactly as above for Families 2-9 and  Group 34. In fact, one sees directly that $(j,1)$ is not isotopic to the identity only when $L$ has order at least four. The case $(1,j)$ is completely analogous by switching the roles of $L$ and $R$. The case of $(j,j)$ follows the same argument at the beginning of the proof, and
it turns out that the only group of order $2$ or $4$ in which $(j,j)$ is not isotopic to the identity is $(C_4/C_2,C_4/C_2).$
\end{proof}

This concludes the proof of Theorem 1. We conclude by remarking that the restriction to Families 1-9 and 34 is essential. 
In fact, besides the exceptions $(C_4/C_1,C_4/C_1)$ and $(C_4/C_2,C_4/C_2)$ of Lemma \ref{lemma trivial action}, in several other cases $\alpha\circ\iota$ is not injective, even in the orientation-preserving case.  For example we suppose that $G$ is the group  $31'$ of Table~\ref{subgroup}; its normalizer in $\SO(4)$ is the group $31$ (see Case 6 in Subsection~\ref{o-p i}). The group  $G$ is isomorphic to the alternating group on five letters and its normalizer  is isomorphic to $G\times \mathbb{Z}_2$. Each isometry of $S^3/G$ induces  a trivial action on $G$. We note that  $\Out(G)$ is not trivial, in fact is isomorphic to  $\mathbb{Z}_2.$
Similarly, $\alpha\circ\iota$ fails to be injective also for the infinite family $11'$.

%We will now show that our claim holds. 
%If $f \in \Isom(\OO)$, $\alpha\circ\iota(f)= $

%Recall that the isometry group of $\OO$ is isomorphic (as a Lie group) to $\Norm_{\SO4}(G)/G$. Let $\tilde G=\Phi^{-1}(G)$, where we recall that $\Phi$ is the $2$-$1$ map from $S^3\times S^3$ to $\SO4$, and denote $L$ and $R$ the left and right projections of $\tilde G$ to $S^3$. Since the element $(-1,-1)$ is central in $S^3\times S^3$, it turns out that $\Norm_{\SO4}(G)/G\cong \Norm_{S^3\times S^3}(\tilde G)/\tilde G$. 

%Moreover, it is easily checked that $\Norm_{S^3\times S^3}(\tilde G)$ is a subgroup of $\Norm_{S^3}L\times\Norm_{S^3}R$. Since $\tilde G$ has finite index in $L\times R$, $\Norm_{S^3\times S^3}(\tilde G)$ consists of path components of $\Norm_{S^3}L\times\Norm_{S^3}R$. In particular the path component of the identity in $\Norm_{S^3}L\times\Norm_{S^3}R$ actually lies in $\Norm_{S^3\times S^3}(\tilde G)$.

%Every element $g\in \Norm_{\SO4}(G)$ acts on $G$ by conjugation. Hence it suffices to show that for every pair $(p,q)\in S^3\times S^3$, if $(p,q)$ induces an inner automorphism on $\tilde G$, then $(p,q)$ is in the connected component of the identity in $\Norm_{S^3\times S^3}(\tilde G)$.

%\cleardoublepage
\bibliographystyle{alpha}
\bibliography{ms-bibliography}

\end{document}